\documentclass[a4paper,11pt,reqno]{amsart}

\usepackage[latin1]{inputenc}

\usepackage{
amssymb,
tikz,
xcolor
}

\usepackage[hidelinks]{hyperref}

\usepackage{geometry}
\newtheorem{thm}{Theorem}[section]
\newtheorem{lemma}[thm]{Lemma}
\newtheorem{prop}[thm]{Proposition}
\newtheorem{cor}[thm]{Corollary}

\theoremstyle{remark}
\newtheorem{remark}[thm]{Remark}

\theoremstyle{definition}
\newtheorem{defi}[thm]{Definition}

\newcommand{\R}{\mathbb{R}}
\newcommand{\C}{\mathbb{C}}

\newcommand{\mE}{\mathrm{E}}
\newcommand{\mV}{\mathrm{V}}
\newcommand{\cM}{\mathcal{M}}
\newcommand{\cK}{\mathcal{K}}
\newcommand{\cF}{\mathcal{F}}
\newcommand{\cD}{\mathcal{D}}
\newcommand{\cL}{\mathcal{L}}
\newcommand{\cG}{\mathcal{G}}
\newcommand{\cO}{\mathcal{O}}
\renewcommand\phi{\varphi}
\newcommand{\eps}{\varepsilon}
\newcommand{\vv}{\mathrm{v}}
\newcommand{\dx}{\,dx}
\newcommand{\dxe}{\,dx_e}
\newcommand{\dom}{\mathrm{dom}}
\renewcommand{\geq}{\geqslant}
\renewcommand{\leq}{\leqslant}
\newcommand{\Sum}{\displaystyle \sum}
\newcommand{\f}[2]{\frac{#1}{#2}}

\renewcommand{\geq}{\geqslant}
\renewcommand{\leq}{\leqslant}

\newcommand\gH{{\mathfrak{H}}}
\newcommand{\gG}{{\Gamma}}

\newcommand\cH{{\mathcal{H}}}

\newcommand\cN{{\mathfrak{N}}}

\newcommand\I{{\rm{i}}}

\newcommand{\be}{\begin{equation}}
\newcommand{\ee}{\end{equation}}

\tikzstyle{nodino}=[circle,draw,fill,inner sep=0pt,minimum size=0.5mm]
\tikzstyle{infinito}=[circle,inner sep=0pt,minimum size=0mm]
\tikzstyle{nodo}=[circle,draw,fill,inner sep=0pt, minimum size=0.5*width("k")]
\tikzstyle{nodo_vuoto}=[circle,draw,inner sep=0pt, minimum size=0.5*width("k")]
\usetikzlibrary{graphs}
\tikzset{every loop/.style={min distance=10mm,in=300,out=240,looseness=10}}
\tikzset{place/.style={circle,thick,draw=blue!75,fill=blue!20,minimum
size=6mm}}
\tikzset{place2/.style={circle,thick,draw=red!75,fill=red!20,minimum
size=6mm}}

\title[NLDE on graphs with localized nonlinearities]{Nonlinear Dirac equation on graphs with localized nonlinearities: bound states and nonrelativistic limit}

\author[W. Borrelli]{William Borrelli}
\address{Universit\'e Paris-Dauphine, PSL Research University, CNRS, UMR 7534, CEREMADE, F-75016 Paris, France.} 
\email{borrelli@ceremade.dauphine.fr}

\author[R. Carlone]{Raffaele Carlone}
\address{Universit\`{a} ``Federico II'' di Napoli, Dipartimento di Matematica e Applicazioni ``R. Caccioppoli'', MSA, via Cinthia, I-80126, Napoli, Italy.} 
\email{raffaele.carlone@unina.it}

\author[L. Tentarelli]{Lorenzo Tentarelli}
\address{Sapienza Universit\`a di Roma, Dipartimento di Matematica, P.le Aldo Moro, 5 , 00185, Roma, Italy.} 
\email{tentarelli@mat.uniroma1.it}

\begin{document}

\begin{abstract}
 In this paper we study the nonlinear Dirac (NLD) equation on noncompact metric graphs with localized Kerr nonlinearities, in the case of Kirchhoff-type conditions at the vertices. Precisely, we discuss existence and multiplicity of the bound states (arising as critical points of the NLD action functional) and we prove that, in the $L^2$-subcritical case, they converge to the bound states of the NLS equation in the nonrelativistic limit.
\end{abstract}

\maketitle

%%%%%%%%%%%%%%%%%%%%%%%%%%%%%%%%%%%%%%%%%%%%%%%%%%%%%%%%%%%%%%%%%%%%%%%%%%%%%%%%%%%%%%%
%%%%%%%%%%%%%%%%%%%%%%%%%%%%%%%%%%%%%%%%%%%%%%%%%%%%%%%%%%%%%%%%%%%%%%%%%%%%%%%%%%%%%%%
%%%%%%%%%%%%%%%%%%%%%%%%%%%%%%%%%%%%%%%%%%%%%%%%%%%%%%%%%%%%%%%%%%%%%%%%%%%%%%%%%%%%%%%

\section{Introduction}
\label{sec-intro}

The investigation of evolution equations on metric graphs (see Section \ref{sec-setting} below for details) has become very popular nowadays as they are assumed to represent effective models for the study of the dynamics of physical systems confined in branched spatial domains. A specific attention has been addressed to the \emph{focusing} nonlinear Schr\"odinger (NLS) equation, i.e.
\begin{equation}
 \label{eq-NLStime}
 \imath\dot{v}=-v''-|v|^{p-2}\,v,\qquad p\geq2,
\end{equation}
with suitable vertex conditions, as it is supposed to well approximate (for $p=4$) the behavior of Bose-Einstein condensates in ramified traps (see, e.g., \cite{GW-PRE} and the references therein).

From the mathematical point of view, the discussion has been mainly focused on the study of the stationary solutions of \eqref{eq-NLStime}, namely functions of the form $v(t,x)=e^{-i\lambda t}\,u(x)$, with $\lambda\in\R$, that solve the stationary version of \eqref{eq-NLStime}, i.e.
\[
 -u''-|u|^{p-2}\,u=\lambda u\,,
\]
with vertex conditions of $\delta$-type. In particular, the most investigated subcase has been that of the \emph{Kirchhoff} vertex conditions, which impose at each vertex:
\begin{itemize}
 \item[(i)] continuity of the function (for details see \eqref{eq-kirch1}),\\[-.3cm]
 \item[(ii)] ``balance'' of the derivatives (for details see \eqref{eq-kirch2}).
\end{itemize}
For a short bibliography limited to the case of noncompact metric graphs with a finite number of edges (which is the framework discussed in the paper) we refer the reader to, e.g., \cite{AST-CVPDE,
AST-JFA,
AST-CVPDE2,
AST-CMP,
CFN-PRE,
CFN-Non,
LLS-JMAA,
NPS-Non,
NRS-JDE} and the references therein.

Following \cite{GSD-PRA,N-RSTA}, also a simplified version of this model has recently gained a particular attention: the case of a nonlinearity localized on the \emph{compact core} $\cK$ of the graph (which is the subgraph consisting of all the bounded edges); namely,
\begin{equation}
 \label{eq-NLSconc}
 -u''-\chi_{_\cK}|u|^{p-2}\,u=\lambda u
\end{equation}
with Kirchhoff vertex conditions and $\chi_{_\cK}$ denoting the characteristic function of $\cK$. This problem has been studied in the $L^2$-subcritical case in \cite{ST-JDE,ST-NA,T-JMAA}, while some new results on the $L^2$-critical case have been presented in \cite{DT-OTAA,DT-p} (for a general overview see also \cite{BCT-p}).

%%%%%%%%%%%%%%%%%%%%%%%%%%%%%%%%%%%%%%%%%%%%%%%%%

\begin{remark}
 We also mention some interesting results on the problem of the bound states on compact graphs. For a purely variational approach we recall, e.g., \cite{D-JDE,CDS-MJM}, whereas for a bifurcation approach we recall, e.g., \cite{MP-AMRX}.
\end{remark}

As a further development, in the last years also the study of the Dirac operator on metric graphs has generated a growing interest (see, e.g., \cite{ALTW-IEOT,
BH-JPA,
BT-JMP,
P}). In particular, \cite{SBMK-JPA} proposed (although in the simplified case of the \emph{infinite 3-star} graph, depicted in Figure \ref{fig-3star}) the study of the \emph{nonlinear} Dirac equation on networks, namely \eqref{eq-NLStime} with the laplacian replaced by the Dirac operator
\begin{equation}
 \label{eq-dirac_formal}
 \cD:=-\imath c\frac{d}{dx}\otimes\sigma_{1}+mc^{2}\otimes\sigma_{3},
\end{equation}
where $m>0$ represents the \emph{mass} of the generic particle of the system and $c>0$ represents the ``\emph{speed of light}''.

%%%%%%%%%%%%%%%%%%%%%%%%%%%%%%%%%%%%%%%%%%%%%%%%
\begin{figure}
 \centering
 \begin{tikzpicture}[xscale= 0.4,yscale=0.4]
 \node at (0,0) [nodo] (00) {};
 \node at (6,0) [infinito] (60) {};
 \node at (7,0) [infinito] (70) {};
 \node at (-4,4) [infinito] (-44) {};
 \node at (-5,5) [infinito] (-4545) {};
 \node at (-4,-4) [infinito] (-4-4) {};
 \node at (-5,-5) [infinito] (-45-45) {};

 \draw [-] (00) -- (60);
 \draw [dashed] (70) -- (60);
 \draw [-] (00) -- (-44);
 \draw [dashed] (-44) -- (-4545);
 \draw [-] (00) -- (-4-4);
 \draw [dashed] (-4-4) -- (-45-45);
 \end{tikzpicture}
 %%%%%%%%%%%%%%%%%%%%%%%%%%%%%%%%%%%%%%%%%%%%%%%%%
 \caption{infinite 3-star graph.}
 \label{fig-3star}
\end{figure}
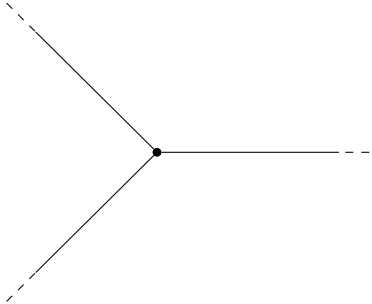
%%%%%%%%%%%%%%%%%%%%%%%%%%%%%%%%%%%%%%%%%%%%%%%%%

\begin{remark}
In fact, $c$ is the \emph{speed of light} only in truly relativistic models, whereas in the present case it should be rather considered as a phenomenological parameter depending on the model under study. Nevertheless, for the sake of simplicity, we will refer to it as ``relativistic parameter'' or ``speed of light'' throughout the paper.
\end{remark}

 Here $\sigma_1$ and $\sigma_3$ are the so-called \emph{Pauli matrices}, i.e.
\begin{equation}
 \label{eq-pauli}
 \sigma_1:=\begin{pmatrix}
  0 & 1 \\
  1 & 0
 \end{pmatrix}
 \qquad\text{and}\qquad
 \sigma_3:=\begin{pmatrix}
  1 & 0 \\
  0 & -1
 \end{pmatrix}\,,
\end{equation}
and with the wave function $v$ replaced by the spinor $\chi:=(\chi^1,\chi^2)^T$. Precisely, \cite{SBMK-JPA} suggests the study again of the stationary solutions, that is $\chi(t,x)=e^{-i\omega t}\,\psi(x)$, with $\omega\in\R$, that solve
\begin{equation}
 \label{eq-NLD}
 \cD\psi-|\psi|^{p-2}\,\psi=\omega \psi\,.
\end{equation}

The attention recently attracted by the linear and the nonlinear Dirac equations is due to their applications, as effective equations, in many physical models, as in solid state physics and nonlinear optics \cite{HC-PD,HC-NJP}. 

While originally the NLDE appeared as a field equation for relativistic interacting fermions \cite{LKG-PRD}, then it was used in particle physics  to simulate features of quark confinement, acoustic physics, and in the context of Bose-Einstein condensates \cite{HC-NJP}.

Recently, it also appeared that some properties of physical models, as thin carbon structures, are well described using as an effective equation for non-relativistic electronic properties , the Dirac  equation. We mention, thereupon, the seminal papers by Fefferman and Weinstein \cite{FW-JAMS,FW-CMP}, the work of Arbunich and Sparber \cite{AS-JMP} (where a rigorous justification of linear and nonlinear equations in two-dimensional honeycomb structures is given) and the referenced therein. In addition, we recall that the existence of stationary solutions for cubic and Hartree-type Dirac equations for honeycomb structures and graphene samples has been investigated in \cite{B-JMP,B-JDE,B-CVPDE}; whereas, for an overview on global existence results for one dimensional NLDE we refer to \cite{CCNP-SIMA,P-RIMS}.

On the other hand, in the context of metric graphs the interest for the Nonlinear Dirac equation arises in the analysis of effective models of condensed matter physics and field theory (\cite{SBMK-JPA}). Moreover, Dirac solitons in networks may be realized in optics, in atomic physics, etc. (see again \cite{SBMK-JPA} and the references therein).

\medskip
In this paper, we discuss the case of \eqref{eq-NLD} with localized nonlinearity (or, equivalently, the Dirac analogous of \eqref{eq-NLSconc}), namely
\[
 \cD\psi-\chi_{_\cK}|\psi|^{p-2}\,\psi=\omega \psi\,.
\]
The reduction to this simplified model arises as one assumes that the nonlinearity affects only the compact core of the graph. This idea was originally exploited in the case of Schr\"odinger equation in \cite{GSD-PRA} and it represents a preliminary step toward  the investigation of the case with the ``extended'' nonlinearity, i.e. \eqref{eq-NLD}, which will be discussed in a forthcoming paper.

\medskip
It is finally worth stressing that, as for the Schr\"odinger case, the operator $\cD$ needs some suitable vertex conditions, which make the operator self-adjoint. In this paper, we limit ourselves to the discussion of those conditions that converge to the Kirchhoff ones in the nonrelativistic limit, and that we call \emph{Kirchhoff-type}. The reason is that they identify (as well as Kirchhoff for Schr\"odinger) the \emph{free case}; namely, the case in which there are no attractive or repulsive effects at the vertices, which then play the role of mere junctions between the edges.

Roughly speaking these conditions ``split'' the requirements of Kirchhoff conditions: the continuity condition is imposed only on the first component of the spinor, while the second component (in place of the derivative) has to satisfy a ``balance'' condition (see \eqref{eq-kirchtype1}$\&$\eqref{eq-kirchtype2}).

\medskip
The paper is organized as follows:
\begin{itemize}
 \item[(i)] in Section \ref{sec-res} we briefly recall some basics on metric graphs and on the properties of the Dirac operator with Kirchhoff-type vertex conditions, and then we state the main results of the paper (Section \ref{sec-main}):\\[-.4cm]
 \begin{itemize}
 \item[-] existence and multiplicity of the bound states (Theorem \ref{thm-bound});\\[-.4cm]
 \item[-] nonrelativistic limit for the bound states (Theorem \ref{thm-limit});\\[-.3cm]
 \end{itemize}
 \item[(ii)] in Section \ref{sec-bound} we show the proof of Theorem \ref{thm-bound};\\[-.3cm]
 \item[(iii)] in Section \ref{sec-limit} we show the proof of Theorem \ref{thm-limit};\\[-.3cm]
 \item[(iv)] in Appendix \ref{sec-linear} we discuss more in details the properties of the Dirac operator with Kirchhoff-type conditions on metric graphs, while Appendix \ref{sec-formdomain} deals with the definition of the form domain of the Dirac operator.
\end{itemize}

\bigskip
\bigskip
\noindent\textbf{Acknowledgements}

\medskip
\noindent We wish to thank Eric S\'{e}r\'{e} for fruitful discussions.

%%%%%%%%%%%%%%%%%%%%%%%%%%%%%%%%%%%%%%%%%%%%%%%%%%%%%%%%%%%%%%%%%%%%%%%%%%%%%%%%%%%%%%%
%%%%%%%%%%%%%%%%%%%%%%%%%%%%%%%%%%%%%%%%%%%%%%%%%%%%%%%%%%%%%%%%%%%%%%%%%%%%%%%%%%%%%%%
%%%%%%%%%%%%%%%%%%%%%%%%%%%%%%%%%%%%%%%%%%%%%%%%%%%%%%%%%%%%%%%%%%%%%%%%%%%%%%%%%%%%%%%

\section{Setting and main Results}
\label{sec-res}

In this section we aim at presenting the main results of the paper. However, the statements of Theorem \ref{thm-bound} and Theorem \ref{thm-limit} require some basics on metric graphs and on the Dirac operator.

%%%%%%%%%%%%%%%%%%%%%%%%%%%%%%%%%%%%%%%%%%%%%%%%%%%%%%%%%%%%%%%%%%%%%%%%%%%%%%%%%%%%%%%

\subsection{Metric graphs and functional setting}
\label{sec-setting}

A complete discussion of the definition and the features of metric graphs can be found in \cite{AST-CVPDE,BK,K-WRM} and the references therein. Here we limit ourselves to recall some basic notions.

Throughout, a \emph{metric graph} $\cG=(\mV,\mE)$ is a connected {\em multigraph} (i.e., multiple edges  and self-loops are allowed) with a finite number of edges and vertices. Each edge is a finite or half-infinite segment of line and the edges are glued together at their endpoints (the vertices of $\cG$) according to the topology of the graph (see Figure \ref{fig-gen}). 

Unbounded edges are identified with (copies of) $\R^+ = [0,+\infty)$ and are called half-lines, while bounded edges are identified with closed and bounded intervals $I_e =[0,\ell_e]$, $\ell_e>0$. Each edge (bounded or unbounded) is endowed with a coordinate $x_e$, chosen in the corresponding interval, which has an arbitrary orientation if the interval is bounded, whereas presents the natural orientation in case of a half-line.

%%%%%%%%%%%%%%%%%%%%%%%%%%%%%%%%%%%%%%%%%%%%%%%%%
\begin{figure}
\centering
\begin{tikzpicture}[xscale= 0.5,yscale=0.5]
\node at (0,0) [nodo] (1) {};
\node at (2,0) [nodo] (2) {};
\node at (4,1) [nodo] (3) {};
\node at (2,3) [nodo] (4) {};
\node at (1,2) [nodo] (5) {};
\node at (-1,2) [nodo] (6) {};
\node at (-2,0) [nodo] (7) {};
\node at (-1,3) [nodo] (8) {};
\node at (-2,2) [nodo] (9) {};
\node at (-4,0) [nodo] (10) {};
\node at (-2,-2) [nodo] (11) {};
\node at (0,-2) [nodo] (12) {};
\node at (3,-2) [nodo] (13) {};
\node at (-3,1) [nodo] (14) {};
\node at (8,-2) [infinito] (15) {};
\node at (9,-2) [infinito] (15b) {};
\node at (8,3) [infinito] (16) {};
\node at (9,3) [infinito] (16b) {};
\node at (-9,-2) [infinito] (17) {};
\node at (-10,-2) [infinito] (17b) {};
\node at (-9,4) [infinito] (18) {};
\node at (-10,4.5) [infinito] (18b) {};
%%%%%%%%%%%%%%%%%%%%%%%%%%%%%%%%%%%%%%%%%%%%%%%%%
\draw (4.5,1) circle (0.5cm);
%%%%%%%%%%%%%%%%%%%%%%%%%%%%%%%%%%%%%%%%%%%%%%%%%
\draw [-] (1) -- (2);
\draw [-] (2) -- (3);
\draw [-] (3) -- (4);
\draw [-] (4) -- (5);
\draw [-] (5) -- (2);
\draw [-] (1) -- (13);
\draw [-] (1) -- (12);
\draw [-] (5) -- (6);
\draw [-] (1) -- (6);
\draw [-] (8) -- (6);
\draw [-] (11) -- (12);
\draw [-] (7) -- (12);
\draw [-] (1) -- (7);
\draw [-] (6) -- (7);
\draw [-] (9) -- (7);
\draw [-] (7) -- (11);
\draw [-] (11) -- (10);
\draw [-] (7) -- (10);
\draw [-] (9) -- (10);
\draw [-] (9) -- (6);
\draw [-] (1) to [out=100,in=190] (5);
\draw [-] (1) to [out=10,in=270] (5);
\draw [-] (13) -- (15);
\draw [dashed] (15) -- (15b);
\draw [-] (4) -- (16);
\draw [dashed] (16) -- (16b);
\draw [-] (14) -- (18);
\draw [dashed] (18) -- (18b);
\draw [-] (11) -- (17);
\draw [dashed] (17) -- (17b);
\end{tikzpicture}
%%%%%%%%%%%%%%%%%%%%%%%%%%%%%%%%%%%%%%%%%%%%%%%%%
\caption{a general noncompact metric graph.}
\label{fig-gen}
\end{figure}
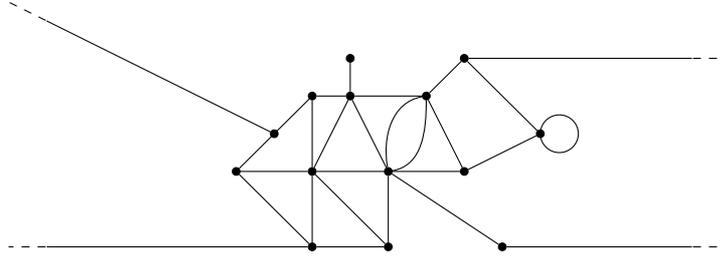
%%%%%%%%%%%%%%%%%%%%%%%%%%%%%%%%%%%%%%%%%%%%%%%%%

As a consequence, the graph $\cG$ is a locally compact metric space, the metric given by the shortest distance along the edges. Clearly, since we assume a finite number of edges and vertices, $\cG$ is \emph{compact} if and only if it does not contain any half-line. A further important notion, introduced in \cite{AST-JFA,ST-JDE} is the following.

\begin{defi}
If $\cG$ is a metric graph, we define its \emph{compact core} $\cK$ as the metric subgraph of $\cG$ consisting of all its bounded edges (see, e.g., Figure \ref{fig-gencomp}). In addition, we denote by $\ell$ the measure of $\cK$, namely
\[
 \ell=\sum_{e\in\cK}\ell_e.
\]
\end{defi}

\begin{figure}
\centering
\begin{tikzpicture}[xscale= 0.5,yscale=0.5]
\node at (0,0) [nodo] (1) {};
\node at (2,0) [nodo] (2) {};
\node at (4,1) [nodo] (3) {};
\node at (2,3) [nodo] (4) {};
\node at (1,2) [nodo] (5) {};
\node at (-1,2) [nodo] (6) {};
\node at (-2,0) [nodo] (7) {};
\node at (-1,3) [nodo] (8) {};
\node at (-2,2) [nodo] (9) {};
\node at (-4,0) [nodo] (10) {};
\node at (-2,-2) [nodo] (11) {};
\node at (0,-2) [nodo] (12) {};
\node at (3,-2) [nodo] (13) {};
\node at (-3,1) [nodo] (14) {};
%%%%%%%%%%%%%%%%%%%%%%%%%%%%%%%%%%%%%%%%%%%%%%%%%
\draw (4.5,1) circle (0.5cm);
%%%%%%%%%%%%%%%%%%%%%%%%%%%%%%%%%%%%%%%%%%%%%%%%%
\draw [-] (1) -- (2);
\draw [-] (2) -- (3);
\draw [-] (3) -- (4);
\draw [-] (4) -- (5);
\draw [-] (5) -- (2);
\draw [-] (1) -- (13);
\draw [-] (1) -- (12);
\draw [-] (5) -- (6);
\draw [-] (1) -- (6);
\draw [-] (8) -- (6);
\draw [-] (11) -- (12);
\draw [-] (7) -- (12);
\draw [-] (1) -- (7);
\draw [-] (6) -- (7);
\draw [-] (9) -- (7);
\draw [-] (7) -- (11);
\draw [-] (11) -- (10);
\draw [-] (7) -- (10);
\draw [-] (9) -- (10);
\draw [-] (9) -- (6);
\draw [-] (1) to [out=100,in=190] (5);
\draw [-] (1) to [out=10,in=270] (5);
\end{tikzpicture}
%%%%%%%%%%%%%%%%%%%%%%%%%%%%%%%%%%%%%%%%%%%%%%%%%
\caption{the compact core of the graph in Figure \ref{fig-gen}.}
\label{fig-gencomp}
\end{figure}
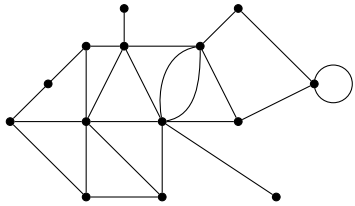
%%%%%%%%%%%%%%%%%%%%%%%%%%%%%%%%%%%%%%%%%%%%%%%%%

\medskip
A function $u:\cG\to\C$ can be regarded as a family of functions $(u_e)$, where $u_e:I_e\to\C$  is the restriction of $u$ to the edge (represented by) $I_e$. The usual $L^p$ spaces can be defined in the natural way, with norm
\[
 \|u\|_{L^p(\cG)}^p := \sum_{e\in\mE} \| u_e\|_{L^p(I_e)}^p,\quad \text{if }p\in[1,\infty),\qquad\text{and}\qquad\|u\|_{L^\infty(\cG)} := \max_{e\in\mE} \| u_e\|_{L^\infty(I_e)},
\]
while $H^1(\cG)$ is the space of functions $u=(u_e)$ such that $u_e\in  H^1(I_e)$ for every edge $e\in\mE$, with norm
\[
 \|u\|_{H^1(\cG)}^2 = \|u'\|_{L^2(\cG)}^2 + \|u\|_{L^2(\cG)}^2
\]
(and in this way one can also define $H^2(\cG)$, $H^3(\cG)$, etc $\dots$). Consistently, a spinor $\psi=(\psi^1,\psi^2)^T:\cG\to\C^2$ is a family of 1d-spinors
\[
 \psi_e=\begin{pmatrix}\psi_e^1\\[.2cm] \psi_e^2\end{pmatrix}:I_e\longrightarrow\C^{2},\qquad \forall e\in\mE,
\]
and thus
\[
L^{p}(\cG,\C^{2}):= \bigoplus_{e\in\mE} L^{p}(I_e)\otimes\C^{2},
\]
endowed with the norm
\[
 \Vert\psi\Vert_{L^{p}(\cG,\C^{2})}^p:=\Sum_{e\in\mE}\Vert \psi_e\Vert_{L^{p}(I_e)}^p,\quad \text{if }p\in[1,\infty),\qquad\text{and}\qquad \Vert\psi\Vert_{L^{\infty}(\cG,\C^2)}:=\max_{e\in\mE}\Vert\psi_e\Vert_{L^{\infty}(I_e)},
\]
while
\[
H^{1}(\cG,\C^{2}):= \bigoplus_{e\in\mE} H^{1}(I_e)\otimes\C^{2}
\]
endowed with the norm
\[
 \Vert\psi\Vert_{H^{1}(\cG,\C^{2})}^2:=\Sum_{e\in\mE}\Vert \psi_e\Vert_{H^{1}(I_e)}^2
\]
(and so on for $H^{2}(\cG,\C^{2})$, $H^{3}(\cG,\C^{2})$, etc $\dots$). Equivalently, one can say that $L^{p}(\cG,\C^{2})$ is the space of the spinors such that $\psi^1,\psi^2\in L^p(\cG)$, with
\[
\begin{array}{l}
 \displaystyle \Vert\psi\Vert_{L^{p}(\cG,\C^{2})}^p:=\Vert\psi^1\Vert_{L^{p}(\cG)}^p+\Vert\psi^2\Vert_{L^{p}(\cG)}^p,\quad \text{if }p\in[1,\infty),\\[.5cm]
 \displaystyle \Vert\psi\Vert_{L^{\infty}(\cG,\C^2)}:=\max\left\{\Vert\psi^1\Vert_{L^{\infty}(\cG)},\Vert\psi^2\Vert_{L^{\infty}(\cG)}\right\}, 
\end{array}
 \]
 and that $H^{1}(\cG,\C^{2})$ is the space of the spinors such that $\psi^1,\psi^2\in H^1(\cG)$, with
\[
 \Vert\psi\Vert_{H^{1}(\cG,\C^{2})}^2:=\Vert\psi^1\Vert_{H^{1}(\cG)}^2+\Vert\psi^2\Vert_{H^{1}(\cG)}^2.
 \]
 
\begin{remark}
 The usual definition of the space $H^1(\cG)$ consists also of a global continuity requirement, which forces all the components of a function that are incident to a vertex to assume the same value at that vertex. However, for the aims of this paper it is worth keeping this global continuity notion separate and introduce it when it is actually required (see \eqref{eq-kirch1}).
\end{remark}

%%%%%%%%%%%%%%%%%%%%%%%%%%%%%%%%%%%%%%%%%%%%%%%%%%%%%%%%%%%%%%%%%%%%%%%%%%%%%%%%%%%%%%%

\subsection{The Dirac operator with Kirchhoff-type conditions}

The expression given by \eqref{eq-dirac_formal} of the Dirac operator on a metric graph is purely formal, since it does not clarify what happens at the vertices of the graph, given that the derivative $\f{d}{dx}$ is well defined just in the interior of the edges.

As well as for the laplacian in the Schr\"odinger case, the way to give a rigorous meaning to \eqref{eq-dirac_formal} is to find suitable self-adjoint realizations of the operator. However, an extensive discussion of all the possible self-adjoint realizations of the Dirac operator on graphs goes beyond the aims of this paper. Throughout, we limit ourselves to the case of the Kirchhoff-type conditions (introduced in \cite{SBMK-JPA}), which represent the free case for the Dirac operator. For more details on self-adjoint extensions of the Dirac operator on metric graphs we refer the reader to \cite{BT-JMP,P}. We also mention \cite{HK-JMAA}, where boundary conditions for 1-D Dirac operators are studied for a model of quantum wires.
 
\begin{defi}
\label{defi-dirac}
Let $\cG$ be a metric graph and let $m,c>0$. We call \emph{Dirac operator} with Kirchhoff-type vertex conditions the operator $\cD:L^2(\cG,\C^2)\to L^2(\cG,\C^2)$ with action
\begin{equation}
\label{eq-actionD}
\cD_{|_{I_e}}\psi=\cD_e\psi_e:=-\imath c\,\sigma_{1}\psi_e'+mc^{2}\,\sigma_{3}\psi_e,\qquad\forall e\in\mE,
\end{equation}
$\sigma_1,\sigma_3$ being the matrices defined in \eqref{eq-pauli}, and domain
\begin{equation}\label{eq-dirac_domain}
 \dom(\cD):=\left\{\psi\in H^1(\cG,\C^2):\psi\text{ satisfies \eqref{eq-kirchtype1} and \eqref{eq-kirchtype2}}\right\},
\end{equation}
where
\begin{gather}
 \label{eq-kirchtype1}
 \psi_e^{1}(\vv)=\psi^{1}_{f}(\vv),\qquad\forall e,f\succ\vv,\qquad\forall \vv\in\cK,\\[.4cm]
 \label{eq-kirchtype2}
 \sum_{e\succ v}\psi^{2}_{e}(\vv)_{\pm}=0,\qquad\forall \vv\in\cK,
\end{gather}
``$e\succ v$'' meaning that the edge $e$ is incident at the vertex $\vv$ and $\psi^{2}_{e}(\vv)_{\pm}$ standing for $\psi^{2}_{e}(0)$ or $-\psi^{2}_{e}(\ell_e)$ according to whether $x_e$ is equal to $0$ or $\ell_e$ at $\mathrm{v}$.
\end{defi}

\begin{remark}
 Note that the operator $\cD$ actually depends of the parameters $m,c$, which represent (as pointed out in Section \ref{sec-intro}) the mass of the generic particle and the speed of light (respectively). For the sake of simplicity we omit this dependence unless it be necessary to avoid misunderstandings.
\end{remark}

The basic properties of the operator \eqref{eq-dirac_formal} with the above conditions are summarized in the following

\begin{prop}
\label{spectrumkirchoff}
The Dirac operator $\cD$ introduced by Definition \ref{defi-dirac} is self-adjoint on $L^{2}(\cG,\C^{2})$. In addition, its spectrum is
\begin{equation}
 \label{eq-sp_D}
 \sigma(\cD)=(-\infty,-mc^2]\cup[mc^2,+\infty).
\end{equation}
\end{prop}

\noindent The discussion of the proof of Proposition \ref{spectrumkirchoff} is briefly presented in Appendix \ref{sec-linear}.

\begin{remark}
 Observe that the self-adjointness of $\cD$ follows directly from the main result of \cite{BT-JMP}, which holds for a wide class of linear vertex conditions.
\end{remark} 

\subsection{The associated quadratic form}\label{sec-quadraticform}

The standard cases of the Dirac operator $\R^d$, with $d=1,2,3$, do not actually require any further remark on the associated quadratic form, which can be easily defined using the Fourier transform (see e.g. \cite{ES-CMP}). Unfortunately, in the framework of the (noncompact) metric graphs this tool is not available and hence it is necessary to resort to the \emph{Spectral Theorem}, which represents a classical, but more abstract way, to diagonalize the operator and, consequently, define the associated quadratic form $\mathcal{Q}_\cD$ and its domain $\dom(\mathcal{Q}_\cD)$ as, for instance,
\[
 \dom(\mathcal{Q}_\cD):=\bigg\{\psi\in L^2(\cG,\C^2):\int_{\sigma(\cD)}|\nu|\,d\mu^\cD_\psi(\nu)\bigg\},\qquad \mathcal{Q}_\cD(\psi):=\int_{\sigma(\cD)}\nu\,d\mu^\cD_\psi(\nu),
\]
where $\mu^\cD_\psi$ denotes the spectral measure associated with $\cD$ and $\psi$.

Unfortunately, this definition is not the most suitable for the purposes of the paper. An alternative way to define the form domain of $\cD$ (that is, $\dom(\mathcal{Q}_\cD)$) is to use the well known \emph{Real Interpolation Theory} \cite{AF,A-p}. Here we just mention some basics, referring to Appendix \ref{sec-formdomain} for some further details.

Define the space
\begin{equation}\label{interpolateddomain}
Y:=\left[L^{2}(\cG,\C^{2}),\dom(\cD)\right]_{\frac{1}{2}},
\end{equation}
namely the interpolated space of order $1/2$ between $L^2$ and the domain of the Dirac operator. First, we note that $Y$ is a closed subspace of
\[
H^{1/2}(\cG,\C^2):=\bigoplus_{e\in\mE} H^{1/2}(I_{e})\otimes\C^{2},
\]
with respect to the norm induced by $H^{1/2}(\cG,\C^2)$. Indeed, $\dom(\cD)$ is clearly a closed subspace of $H^1(\cG,\C^2)$ and there results (arguing edge by edge) that 
\[
 H^{1/2}(\cG,\C^2)=\left[L^{2}(\cG,\C^{2}),H^1(\cG,\C^2)\right]_{\frac{1}{2}},
\]
so that the closedness of $Y$ follows by the very definition of interpolation spaces . As a consequence, by Sobolev embeddings there results that
 \be\label{sobolev}
  Y\hookrightarrow L^{p}(\cG,\C^{2}), \qquad\forall p\in[2,\infty),
 \ee
and that, in addition, the embedding in $L^{p}(\cK,\C^{2})$ is compact, due to the compactness of $\cK$.

On the other hand, there holds (see Appendix \ref{sec-formdomain})
\begin{equation}
 \label{eq-formeq}
 \dom(\mathcal{Q}_\cD)=Y,
\end{equation}
and hence the form domain inherits all the properties pointed out before, which are in fact crucial in the following of the paper.

\medskip
Finally, for the sake of simplicity (and following the literature on the NLD equation), we denote throughout the form domain by $Y$, in view of \eqref{eq-formeq}, and
\[
 \mathcal{Q}_\cD(\psi)=\frac{1}{2}\int_{\cG}\langle\psi,\cD\psi \rangle\dx,\qquad\text{and}\qquad\mathcal{Q}_\cD(\psi,\varphi)=\frac{1}{2}\int_{\cG}\langle\psi,\cD\varphi \rangle\dx,
\]
with $\langle\,\cdot\,,\,\cdot\,\rangle$ denoting the euclidean sesquilinear product of $\C^2$, since this does not give rise to misunderstandings. In particular, as soon as $\psi$ and/or $\varphi$ are smooth enough (e.g., if they belong to the operator domain) the previous expressions gain an actual meaning as Lebesgue integrals.

We also recall that in the sequel we denote by $\langle\,\cdot\,|\,\cdot\,\rangle$ duality pairings (the function spaces involved being clear from the context).

\begin{remark}
 Note that the the combination between Spectral Theorem and Interpolation Theory is (to the best of our knowledge) the sole possibility to define the quadratic form, since also classical duality arguments fail due to the fact that it is not true in general that $H^{-1/2}(\cG,\C^2)$ is the topological dual of $H^{1/2}(\cG,\C^2)$ (due to the presence of bounded edges).
\end{remark}

%%%%%%%%%%%%%%%%%%%%%%%%%%%%%%%%%%%%%%%%%%%%%%%%%%%%%%%%%%%%%%%%%%%%%%%%%%%%%%%%%%%%%%%

\subsection{Main results}
\label{sec-main}

We can now state the main results of the paper. Preliminarily, we give the definition of the bound states of the NLD and of the NLS equations on noncompact metric graphs with localized nonlinearities.

\begin{defi}[Bound states of the NLDE]
 Let $\cG$ be a noncompact metric graph with nonempty compact core $\cK$ and let $p>2$. Then, a \emph{bound state of the NLDE} with Kirchhoff-type vertex conditions and nonlinearity localized on $\cK$ is a spinor $0\not\equiv\psi\in\dom(\cD)$ for which there exists $\omega\in\R$ such that
 \begin{equation}
  \label{eq-NLDbound}
  \cD_e\psi_e-\chi_{_\cK}|\psi_e|^{p-2}\psi_e=\omega\psi_e,\qquad\forall e\in\mE,
 \end{equation}
 with $\chi_{_\cK}$ the characteristic function of the compact core $\cK$.
\end{defi}

\begin{defi}[Bound states of the NLSE]
 \label{defi-NLSE}
 Let $\cG$ be a noncompact metric graph with nonempty compact core $\cK$, and let $p>2$ and $\alpha>0$. Then, a \emph{bound state of the NLSE} equation with Kirchhoff vertex conditions and focusing nonlinearity localized on $\cK$ is a function $0\not\equiv u\in H^2(\cG)$ that satisfies
 \begin{gather}
 \label{eq-kirch1}
 u_e(\vv)=u_f(\vv),\qquad\forall e,f\succ\vv,\qquad\forall \vv\in\cK,\\[.4cm]
 \label{eq-kirch2}
 \sum_{e\succ v}\f{du_e}{dx_e}(\vv)=0,\qquad\forall \vv\in\cK,
\end{gather}
where $\frac{du_e}{dx_e}(\mathrm{v})$ stands for $u_e'(0)$ or $-u_e'(\ell_e)$ according to whether $x_e$ is equal to $0$ or $\ell_e$ at $\mathrm{v}$, and for which there exists $\lambda\in\R$ such that
 \begin{equation}\label{eq-NLSbound}
  -u_e''-\alpha\chi_{_\cK}|u_e|^{p-2}u_e=\lambda u_e,\qquad\forall e\in\mE.
 \end{equation}
\end{defi}

\begin{remark}
 In the definition above we allow the presence of the parameter $\alpha$, merely in view of the result stated by Theorem \ref{thm-limit}.
\end{remark}

\begin{remark}
 We recall that conditions \eqref{eq-kirch1}$\&$\eqref{eq-kirch2} make the laplacian self-adjoint on $\cG$ and are called Kirchhoff conditions. We also recall that the parameters $\omega$ and $\lambda$ are usually referred to as \emph{frequencies} of the bound states of the NLDE and NLSE (respectively).
\end{remark}

\begin{thm}[Existence and multiplicity of the bound states]
 \label{thm-bound}
 Let $\cG$ be a noncompact metric graph with nonempty compact core and let $m,c>0$ and $p>2$. Then, for every $\omega\in(-mc^2,mc^2)$ there exists infinitely many (distinct) pairs of bound states of frequency $\omega$ of the NLDE.
\end{thm}

Some comments are in order. First of all, to the best of our knowledge this is the first rigorous result on the stationary solutions of the nonlinear Dirac equation on metric graphs.

On the other hand, some relevant differences can be observed with respect to the Schr\"odinger case. Bound states of Theorem \ref{thm-bound} arise (as we will extensively show in the next section) as critical points of the functional
\[
 \cL(\psi):=\frac{1}{2}\int_{\cG}\langle\psi,(\cD-\omega)\psi \rangle \dx-\frac{1}{p}\int_{\cK}\vert\psi\vert^{p}\dx.
\]
However, due to the spectral properties of $\cD$, the kinetic part of $\cL$ (that is, the quadratic form associated with $\cD$) is unbounded from below even if one constrains the functional to the set of the spinors with $L^2$-norm fixed, in contrast to the NLS functional. As a consequence, no minimization can be performed and, hence, the extensions of the direct methods of calculus of variations developed for the Schr\"odinger case are useless.

Furthermore, such a kinetic part is also strongly indefinite, so that the functional possesses a significantly more complex geometry with respect to the NLS case, thus calling for more technical (albeit classical) tools of Critical Point Theory. 

Finally, the spinorial structure of the problem as well as the implicit definition of the kinetic part of the functional, whose domain is not embedded in $L^\infty(\cG,\C^2)$, prevent the (direct) use of the tools developed for the NLSE on graphs such as, for instance, rearrangements and ``graph surgery''.

In view of these issues, in the proof of Theorem \ref{thm-bound} we rather adapted some techniques from the literature on the NLDE on standard noncompact domains. Anyway, the fact that we are dealing with a nonlinearity localized only on a compact part of the graph makes the study of the geometry of the functional a bit more delicate  as we will see in Lemma \ref{testlinking} (while it clearly simplifies the compactness issues with respect to the extended case). For the same reason, the uniform $H^{1}$-boundedness needed to study the non relativistic limit of the bound states (see below) is achieved in different steps (see Section \ref{sec-limit}).

\medskip
The second (and main) result of the paper, on the other hand, shows the connection between the NLDE and the NLSE, suggested by the physical interpretation of the two models.

Before presenting the statement, we recall that, by the definition of $\cD$, the bound states obtained via Theorem \ref{thm-bound} depend in fact on the speed of light $c$. As a consequence, they should be meant as bound states of frequency $\omega$ of the NLDE \emph{at speed of light $c$}.

\begin{thm}[Nonrelativistic limit of the bound states]
 \label{thm-limit}
 Let $\cG$ be a noncompact metric graph with nonempty compact core, and let $m>0$, $p\in(2,6)$ and $\lambda<0$. Let also $(c_{n}), (\omega_{n})$ be two real sequences such that
 \begin{gather}
 \label{cdiverge}
 0<c_{n},\omega_n\rightarrow+\infty,\\[.2cm]
 \label{eq-smaller}
 \omega_{n}<mc^{2}_{n},\\[.2cm]
 \label{parameterconvergence}
 \omega_n-mc_n^2\longrightarrow\f{\lambda}{m},
\end{gather}
as $n\rightarrow+\infty$. If $\{\psi_n=(\psi_n^1,\psi_n^2)^T\}$ is a bound state of frequency $\omega_n$ of the NLDE \eqref{eq-NLDbound} at speed of light $c_n$, then, up to subsequences, there holds
 \[
  \psi_n^1\longrightarrow u\qquad\text{and}\qquad\psi_n^2\longrightarrow0\qquad\text{in}\quad H^1(\cG),
 \]
 as $n\rightarrow+\infty$, where $u$ is a bound state of frequency $\lambda$ of the NLSE \eqref{eq-NLSbound} with $\alpha=2m$.
\end{thm}

First we recall that the expression ``speed of light $c_n$, with $c_n\to\infty$'' corresponds to the fact that we are investigating the regime where the relativistic parameter $c$ is large. The main interest of Theorem \ref{thm-limit} is thus given by the fact that we recover the NLSE model in the limit, providing a mathematical proof of this fact, quite natural from the physical point of view.

On the other hand, the fact that $\alpha=2m$ is consistent with the fact that in the nonrelativistic limit the kinetic (free) part of the hamiltonian of a particle is given by $-\frac{1}{2m}\Delta.$

Moreover, we point out that Theorem \ref{thm-limit}, in contrast to Theorem \ref{thm-bound}, holds only for a fixed range of power exponents, namely the so-called $L^2$-subcritical case $p\in(2,6)$. However, this is the only range of powers for which multiplicity results are known for the NLSE (see \cite{ST-JDE}). On the other hand, these results are parametrized by the $L^2$ norm of the wave function while Theorem \ref{thm-limit} is parametrized by the frequency and hence (in some sense) it presents as a byproduct a new result for the NLSE.

\begin{remark}
 We also mention that Theorem \ref{thm-bound} and Theorem \ref{thm-limit} can be proved, without significant modifications, also in the case of more general nonlinearities, by means of several ad hoc assumptions. We limit ourselves to the power case in this context for the sake of simplicity. 
\end{remark}

%%%%%%%%%%%%%%%%%%%%%%%%%%%%%%%%%%%%%%%%%%%%%%%%%%%%%%%%%%%%%%%%%%%%%%%%%%%%%%%%%%%%%%%%%%%%%%%%%%%%%%%%%%%%%%%%%%%%%%%%%%%%%%%%%%
\section{Existence of infinitely many bound states}
\label{sec-bound}

In this Section we prove Theorem \ref{thm-bound}. Note that, since the parameter $c$ here does not play any role, we set $c=1$ throughout the section. In addition, in the sequel (unless stated otherwise) we always tacitly assume that the mass parameter $m$ is positive, the frequency $\omega\in(-m,m)$, the power of the nonlinearity $p>2$ and that $\cG$ is a noncompact metric graph with nonempty compact core.

%%%%%%%%%%%%%%%%%%%%%%%%%%%%%%%%%%%%%%%%%%%%%%%%%%%%%%%%%%%%%%%%%%%%%%%%%%%%%%%%%%%%%%%

\subsection{Preliminary results}

The first point is to prove that the bound states coincide with the critical points of the $C^{2}$ action functional $\cL:Y\to\R$ defined by
\be
\label{action}
\cL(\psi)=\frac{1}{2}\int_{\cG}\langle\psi,(\cD-\omega)\psi \rangle \dx-\frac{1}{p}\int_{\cK}\vert\psi\vert^{p}\dx.
\ee
Recall that (as $c=1$) the spectrum of $\cD$ is given by 
\be
\label{eq-spectrum2}
\sigma(\cD)=(-\infty,-m]\cup[m,+\infty).
\ee

\begin{prop}
\label{prop-boundcrit}
 A spinor is a bound state of frequency $\omega$ of the NLDE if and only if it is a critical point of $\cL$.
\end{prop}

\begin{proof}
 One can easily see that a bound state of frequency $\omega$ of the NLDE is a critical point of $\cL$. Let us prove, therefore, the converse.
  
  Assume that $\psi$ is a critical point of $\cL$, namely that $\psi\in Y$ and 
 \begin{equation}
  \label{eq-boundweak}
   \langle d\cL(\psi)|\varphi\rangle=\int_\cG\langle\psi,(\cD-\omega)\phi\rangle\dx-\int_\cK|\psi|^{p-2}\langle\psi,\varphi\rangle\dx=0,\qquad\forall \phi\in Y.
 \end{equation}
 Now, for any fixed edge $e\in\mE$, if one chooses
 \begin{equation}
  \label{eq-test}
  \phi=\begin{pmatrix}\phi^1\\[.2cm] 0\end{pmatrix},\qquad\text{with}\quad0\neq\phi^1\in C_0^\infty(I_e)
 \end{equation}
 (namely, $\phi^1$ possesses the sole component $\phi_e^1$, which is a test function of $I_e$), then
 \[
  \imath\int_{I_e}\psi_e^2\,(\overline{\phi}_e^1)'\dxe=\int_{I_e}\underbrace{\left[(m-\omega)\psi_e^1+\chi_{-\cK}|\psi_e|^{p-2}\psi_e^1\right]}_{\in L^2(I_e)}\overline{\phi}_e^1\dxe,
 \]
 so that $\psi_e^2\in H^1(I_e)$ and an integration by parts yields the first line of \eqref{eq-NLDbound}. On the other hand, simply exchanging the role of $\phi^1$ and $\phi^2$ in \eqref{eq-test}, one can see that $\psi_e^1\in H^1(I_e)$ and satisfies the second line of \eqref{eq-NLDbound}, as well.
 
 It is then left to prove that $\psi$ fulfills \eqref{eq-kirchtype1} and \eqref{eq-kirchtype2}. First, fix a vertex $\vv$ of the compact core and choose
 \[
  \dom(\cD)\ni\phi=\begin{pmatrix}\phi^1\\[.2cm] 0\end{pmatrix},\qquad\text{with}\qquad\phi^1(\vv)=1,\qquad\phi(\vv')=0\quad\forall\vv'\in\cK,\,\vv'\neq\vv.
 \]
 Integrating by parts in \eqref{eq-boundweak} and using \eqref{eq-NLDbound}, there results
 \[
  \sum_{e\succ\vv}\phi_e^1(\vv)\psi_e^2(\vv)_{\pm}=0
 \]
 and, hence, $\psi^2$ satisfies \eqref{eq-kirchtype2} (recall the meaning of $\psi_e^2(\vv)_{\pm}$ explained in Definition \ref{defi-dirac}). On the other hand, let $\vv$ be a vertex of the compact core with degree greater than or equal to 2 (for vertices of degree 1 \eqref{eq-kirchtype1} is satisfied for free). Moreover, let
 \[
  \dom(\cD)\ni\phi=\begin{pmatrix}0\\[.2cm] \phi^2\end{pmatrix},\qquad\text{with}\qquad\phi_{e_1}^2(\vv)_{\pm}=-\phi_{e_2}^2(\vv)_{\pm},\qquad\phi_e^2(\vv)=0\quad\forall e\neq e_1,e_2,
 \]
 where $e_1$ and $e_2$ are two edges incident at $\vv$, and $\phi_e^2\equiv0$ on each edge not incident at $v$. Again, integrating by parts in \eqref{eq-boundweak} and using \eqref{eq-NLDbound},
 \[
  \phi_{e_1}^2(\vv)_{\pm}\psi_{e_1}^1(\vv)+\phi_{e_2}^2(\vv)_{\pm}\psi_{e_2}^1(\vv)=0.
 \]
 Then, repeating the procedure for any pair of edges incident at $\vv$ one gets \eqref{eq-kirchtype1}.
 
 Finally, iterating the same arguments on all the vertices one concludes the proof.
\end{proof}

\begin{remark}
In addition to Proposition \ref{prop-boundcrit}, it is worth mentioning that, due to the linear behavior outside the compact core, the bound states are known explicitly on the half-lines. Pecisely, if $e\in\mE$ is a half-line with starting point $\vv$, then
\begin{equation}\label{explicithalflines}
 \left\{\begin{array}{l}
  \displaystyle \psi_e^1(x_e)=-\imath\,\psi_e^2(\vv)\sqrt{\frac{m+\omega}{m-\omega}}\,e^{-\sqrt{m^2-\omega^2}x_e}\\[.6cm]
  \displaystyle \psi_e^2(x_e)=\psi_e^2(\vv)e^{-\sqrt{m^2-\omega^2}x_e}
 \end{array}\right.,\qquad x_e\in[0,\infty).
\end{equation}
\end{remark}

\medskip
The second preliminary step is to prove that the functional $\cL$ possesses a so-called \emph{linking geometry} (\cite{ES-CMP,S}), since this is the main tool in order to prove the existence of Palais-Smale sequences.

Recall that, according to \eqref{eq-spectrum2} we can decompose the form domain $Y$ as the orthogonal sum of the positive and negative spectral subspaces for the operator $\cD$, i.e.
\[
Y=Y^{+}\oplus Y^{-}.
\]
As a consequence, every $\psi\in Y$ can be written as $\psi=P^+\psi+P^-\psi=:\psi^{+}+\psi^{-}$, where $P^{\pm}$ are the orthogonal projectors onto $Y^{\pm}$. In addition one can find an equivalent (but more convenient) norm for $Y$, i.e.
\[
\Vert\psi\Vert:=\Vert\sqrt{\vert\cD\vert}\psi\Vert_{L^{2}},\qquad \forall\psi\in Y.
\]
\begin{remark}
Borel functional calculus for self-ajoint operators \cite[Theorem VIII.5]{RS-I} allows to define the operators $\vert\cD\vert^{\alpha}$, $\alpha>0$, and more general operators of the form $f(\cD)$, where $f$ is a Borel function on $\R$.
\end{remark}
In view of the previous remarks and using again the Spectral theorem, the action functional \eqref{action} can be rewritten as follows
\be\label{rewritten}
\cL(\psi)=\frac{1}{2}(\Vert\psi^{+}\Vert^{2}-\Vert\psi^{-}\Vert^{2})-\f{\omega}{2}\int_{\cG}\vert\psi\vert^{2}-\frac{1}{p}\int_{\cK}\vert\psi\vert^{p}\dx,
\ee
which is the best form in order to prove that $\cL$ has in fact a linking geometry (see e.g. \cite[Section II.8]{S}).

\begin{lemma}\label{testlinking}
For every $N\in\mathbb{N}$ there exist $R=R(N,p)>0$ and an $N$-dimensional space $Z_{N}\subset Y^{+}$ such that
\be
\label{negative}
\cL(\psi)\leq 0,\qquad \forall\psi\in\partial\cM_N,
\ee
where
\[ 
\partial \cM_{N}=\left\{\psi\in\cM_N : \Vert\psi^{-}\Vert=R \quad \text{\emph{or}}\quad \Vert\psi^{+}\Vert=R \right\}.
\]
and
\be
\label{eq-emmeenne}
\cM_{N}:=\left\{\psi\in Y : \Vert\psi^{-}\Vert\leq R \quad \text{\emph{and}}\quad \psi^{+}\in Z_{N} \quad \text{\emph{with}}\quad\Vert\psi^{+}\Vert\leq R \right\}.
\ee
\end{lemma}

\begin{proof}
Let $e$ be a bounded edge, associated with the segment $I_e=[0,\ell_e]$, and let $V$ be the space of the spinors
\[
\eta=\begin{pmatrix}\eta^1\\[.2cm] 0\end{pmatrix},\qquad\text{with}\quad\eta^1\in C_0^\infty(I_e),
\]
which is clearly a subset of $\dom(\cD)$ and hence of $Y$. Moreover, a simple computation shows that 
\be
\label{partepositiva}
\int_{\cG}\langle\eta,\cD\eta\rangle\dx =m\int_{\cG}\vert\eta_{1}\vert^{2}\dx
\ee
and thus, in view of \eqref{rewritten}, if $\eta_1\neq0$ then $\eta^{+}\neq0$.
 
Assume first that $\mathrm{dim} V^+=\infty$, where $V^+=V\cap Y^+$. For every fixed $N\in\mathbb{N}$, choose $N$ linearly independent spinors $\eta^{+}_{1},...,\eta^{+}_{N}\in V^+$ and set $Z_{N}:= \operatorname{span}\{\eta^{+}_{1},...,\eta^{+}_{N}\}$. As a consequence, if $\psi\in\partial\cM_{N}$, then $\psi=\varphi+\xi$ with $\varphi\in Y^{-}$ and $\xi\in Z_{N}$, so that
\[
\cL(\psi)=\cL(\varphi+\xi)=\frac{1}{2}\left(\Vert\xi\Vert^{2}-\Vert\varphi\Vert^{2} \right)-\frac{1}{p}\int_{\cK}\vert\varphi+\xi \vert^{p}\dx.
\]
It is clear that, if $\Vert\varphi\Vert\geq\Vert\xi\Vert$, then
\[
\cL(\varphi+\xi)\leq -\int_{\cK}\vert\varphi+\xi\vert^{p}\dx\leq 0
\]
If, on the contrary, $\Vert\xi\Vert\geq\Vert\varphi\Vert$, then some further effort is required. Since $\psi\in\partial\cM_N$, $\Vert\xi\Vert=R$ and thus
\[
\cL(\varphi+\xi)\leq \frac{R^{2}}{2} - \frac{1}{p}\int_{\cK}\vert\varphi+\xi\vert^{p}\dx.
\]
From the H\"{o}lder inequality
\be\label{holder}
\int_{\cK}\vert\varphi+\xi\vert^{2}\leq \ell^{\frac{p-2}{p}}\left(\int_{\cK}\vert\varphi+\xi\vert^{p}\dx \right)^{\frac{2}{p}}
\ee 
(recall that $\ell=|\cK|$) and hence
\be\label{intermedia}
\cL(\varphi+\xi)\leq \frac{R^{2}}{2} - \frac{\ell^{\frac{p(2-p)}{4}}}{p}\left(\int_{\cK}\vert\varphi+\xi\vert^{2}\dx\right)^{\frac{p}{2}}.
\ee
Now, by definition, $\xi=\sum^{N}_{j=1}\lambda_{j}\eta^{+}_{j}$, for some $\lambda_{j}\in\C$. On the other hand, denoting by $\eta^{-}_{j}$ the spinors such that $\eta^{-}_{j}+\eta^{+}_{j}=:\eta_{j}\in V$, since $\varphi\in Y^{-}$, there results that $\varphi=\varphi^{\perp}+\chi$, with $\chi:=\sum^{N}_{j=1}\lambda_{j}\eta^{-}_{j}$ and $\varphi^{\perp}$ the orthogonal complement of $\chi$ in $Y^{-}$. Therefore, as $\varphi^{\perp}$ is orthogonal to $\chi$ and $\xi$ in $L^2(\cG,\C^2)$,
 \be\label{prima}
 \int_{\cG}\vert\varphi+\xi\vert^{2}\dx=\int_{\cG}\vert\varphi^{\perp}\vert^{2}\dx + \int_{\cG}\vert\xi+\chi\vert^{2}\dx;
 \ee
 while, as  $\xi+\chi=\sum^{N}_{j=1}\lambda_{j}\eta_{j}$ vanishes outside $I\subset\cK$,
 \be\label{seconda}
 \int_{\cG}\vert\varphi+\xi\vert^{2}\dx= \int_{\cG\setminus\cK}\vert\varphi+\xi\vert^{2}\dx+ \int_{\cK}\vert\varphi+\xi\vert^{2}\dx= \int_{\cG\setminus\cK}\vert\varphi^{\perp}\vert^{2}\dx+ \int_{\cK}\vert\varphi+\xi\vert^{2}\dx.
 \ee
 Combining \eqref{prima} and \eqref{seconda} we get
\[
\int_{\cK}\vert \varphi+\xi\vert^{2}\dx=  \int_{\cK}\vert\varphi^{\perp}\vert^{2}\dx+ \int_{\cG}\vert\chi+\xi\vert^{2}\dx
\]
and, plugging into \eqref{intermedia},
\be
\label{eq-quasifinal}
\cL(\varphi+\xi)\leq \frac{R^{2}}{2} - \frac{\ell^{\frac{p(2-p)}{4}}}{p}\left( \int_{\cK}\vert\varphi^{\perp}\vert^{2}\dx+ \int_{\cG}\vert\chi+\xi\vert^{2}\dx\right)^{\frac{p}{2}}\leq\frac{R^{2}}{2} - \frac{\ell^{\frac{p(2-p)}{4}}}{p}\left( \int_{\cG}\vert\chi+\xi\vert^{2}\dx\right)^{\frac{p}{2}}.
\ee
Then, since $\chi$ and $\xi$ are orthogonal by construction and $\chi+\xi$ belongs to a finite dimensional space (so that its $L^{2}$-norm is equivalent to the $Y$-norm), there exists $C>0$ such that 
\[
\cL(\varphi+\xi)\leq \frac{R^{2}}{2}- C\left(\Vert\chi\Vert^{2}+\Vert\xi\Vert^{2} \right)^{\frac{p}{2}}\leq \frac{R^{2}}{2}- C\Vert\xi\Vert^{p} =\frac{R^{2}}{2}- C R^{p}
\]
and thus, for $R$ large, the claim is proved.

Finally, consider the case $\mathrm{dim} V^+<\infty$. As $\mathrm{dim} V=\infty$, we have $\mathrm{dim} V^-=\infty$. On the other hand, there holds $\sigma_2V^-\subset Y^+$ and that $\sigma_2V^+\subset Y^-$, where $\sigma_2$ is the Pauli matrix
\[
 \sigma_{2}=\begin{pmatrix}0&-i\\[.2cm] i&0 \end{pmatrix},
\]
as this matrix anticommutes with the Dirac operator. Therefore (also recalling that $\sigma_2$ in unitary), if one defines $\widetilde{V}=\sigma_2 V$, which consists of spinors of the form
\[
\eta=\begin{pmatrix}0\\[.2cm] \eta^2\end{pmatrix},\qquad\text{with}\quad\eta^2\in C_0^\infty(I_e),
\]
so that $\widetilde{V}^+=\sigma_2V^-$ and $\widetilde{V}^-=\sigma_2V^+$, then (arguing as before) one can prove again \eqref{negative}.
\end{proof}

\begin{lemma}
\label{positivesphere}
There exist $r,\rho>0$ such that 
\[
\inf_{S^{+}_{r}}\cL\geq\rho>0,
\]
where
\begin{equation}
 \label{eq-sfera}
 S^{+}_{r}:=\left\{\psi\in Y^{+}: \Vert\psi\Vert=r \right\}.
\end{equation}
\end{lemma}

\begin{proof}
 The proof is an immediate consequence of the definition of $\cL$ given in \eqref{action}, in view of the fact that $p>2$ and $\omega\in(-m,m)$.
\end{proof}

Finally, we introduce a further representation of the functional $\cL$, which will be useful in the sequel. Preliminarily, note that as
the spectrum of the (self-adjoint) operator $(\cD-\omega)$ is given by 
\[
\sigma(\cD-\omega)=(-\infty,-m-\omega]\cup[m-\omega,+\infty)
\]
(and as $|\omega|<m$), one can define an equivalent norm
\[
\Vert\psi\Vert_{\omega}:=\Vert\sqrt{\vert\cD-\omega\vert}\psi\Vert_{L^{2}},\qquad \forall\psi\in Y,
\]
and the two spectral projectors $P^{\pm}_{\omega}$ on the positive and negative (respectively) spectral subspaces of $(\cD-\omega)$. As a consequence,
\be
\label{eq-projomega}
\psi=P^{+}_{\omega}\psi+P^{-}_{\omega}\psi,\qquad\forall\psi\in Y
\ee
and \eqref{rewritten} can be written as
\[
\cL(\psi)=\frac{1}{2}(\Vert\psi^{+}\Vert_{\omega}^{2}-\Vert\psi^{-}\Vert_{\omega}^{2})-\frac{1}{p}\int_{\cK}\vert\psi\vert^{p}\dx.
\]

%%%%%%%%%%%%%%%%%%%%%%%%%%%%%%%%%%%%%%%%%%%%%%%%%%%%%%%%%%%%%%%%%%%%%%%%%%%%%%%%%%%%%%%

\subsection{Existence and multiplicity of the bound states}

The aim of this section is to prove, for $p>2$, the existence of infinitely many (pairs of) bound states of the NLDE for any frequency $\omega\in(-m,m)$. The techniques used below (such as, for instance, Krasnoselskij genus or pseudo-gradient flows) are well-known in the literature in their abstract setting and can be found for instance in \cite{R-CBMS,S} (see also \cite{ES-CMP} for an application to nonlinear Dirac equations).

\medskip
Recall the definition of Krasnoselskij genus for the subsets of $Y$.
\begin{defi}
Let $\mathcal{A}$ be the family of sets $A\subset Y\backslash\{0\}$ such that $A$ is closed and symmetric (namely, $\psi\in A \Rightarrow -\psi\in A$). For every $A\in\mathcal{A}$, the {\em genus} of $A$ is the natural number defined by
\[
\gamma[A]:=\min\{n\in\mathbb{N}:\exists\varphi:A\rightarrow\R^n\backslash\{0\},\, \varphi\mbox{ continuous and odd}\}.
\]
If no such $\varphi$ exists, then one sets $\gamma[A] = \infty$.
\end{defi}

In addition, one easily sees that the action functional $\cL$ is even, i.e. 
\[
\cL(-\psi)=\cL(\psi),\qquad\forall\psi\in Y.
\]
As a consequence, it is well known (see, e.g., \cite[Appendix]{R-CBMS}) that there exists an odd pseudo-gradient flow $\left( h_{t}\right)_{t\in\R}$ associated with the functional $\cL$, which satisfies some useful properties. This construction is based on well-known arguments and, thus, here we only present an outline of the proof, refering the reader to \cite[Appendix]{R-CBMS} and \cite[Chapter II]{S} for details.

Since the interaction term is concentrated on a compact set $\cK\subset\cG$, the compactness of Sobolev embeddings imply that $h_{t}$ can be chosen of the following form 
 \[
 h_{t}=\Lambda_{t}+K_{t}:[0,\infty)\times Y\longrightarrow Y,
 \]
where $\Lambda_{t}$ is an isomorphism and $K_{t}$ is a compact map, for all $t\geq0$. Moreover, one can also prove that 
\[
\Lambda_{t}:Y^{-}\oplus Y^{+}\longrightarrow Y^{-}\oplus Y^{+}, \qquad\forall t\in\R,
\]
that is, $Y^{\pm}$ are invariant under the action of $\Lambda_{t}$ for all $t\geq0$.

Fix, then, $\varepsilon>0$ such that $\rho-\varepsilon>0$ (with $\rho$ given by Lemma \ref{positivesphere}). Exploiting suitable cut-off functions on the pseudogradient vector field, one can get that, for all $t\geq 0$,
\be\label{cutoffedflow}
h_{t}(\psi)=\psi,\qquad \forall\psi\in \{\varphi\in Y : \cL(\varphi)<\rho-\varepsilon \},
\ee
namely, the level sets of the action below $\rho-\varepsilon$ are not modified by the flow.

In view of these remarks, we can state the following lemma.

\begin{lemma}
\label{genusintersezione}
Let $r,\rho>0$ be as in Lemma \ref{positivesphere}. Then, for every $N\in\mathbb{N}$, there results 
\[
\gamma[h_{t}(S^{+}_{r})\cap\cM_{N}]\geq N, \qquad\forall t\geq0,
\] 
with $S^{+}_{r}$ and $\cM_{N}$ defined by \eqref{eq-sfera} and \eqref{eq-emmeenne}, respectively.
\end{lemma}

\begin{proof}
For each fixed $\psi\in Y$, the function $t\mapsto\cL\circ h_{t}(\psi)$ is increasing. Then Lemma \ref{testlinking} implies that 
\[
h_{t}(S^{+}_{r})\cap\partial\cM_{N}=\emptyset,\qquad \forall t\geq0.
\]
Note, also, that by the group property of the pseudogradient flow 
\[
\left( h_{t}\right)^{-1}=h_{-t}=\Lambda_{-t}+K_{-t},
\]
so that
\[
h_{t}(S^{+}_{r})\cap\cM_{N}=h_{t}\left( S^{+}_{r}\cap h_{-t}\left(\cM_{N} \right)\right).
\]
Then, from \eqref{cutoffedflow}, a degree-theory argument (see e.g. \cite[Section II.8]{S}) shows that 
\[
 S^{+}_{r}\cap h_{-t}\left(\cM_{N} \right)\neq\emptyset.
\]

On the other hand, by \eqref{cutoffedflow} and Lemma \ref{testlinking}, it is easy to see that 
\[
h_{t}(S^{+}_{r})\cap\cM_{N}=h_{t}(S^{+}_{r})\cap (Y^-\oplus Z_{N}),
\]
and thus
\be\label{inverseL}
h_{t}(S^{+}_{r})\cap\cM_{N}=h_{t}\left(S_r^+\cap\big(Y^-\oplus Z_N'+ K_{-t}(Y^-\oplus Z_N)\big) \right),
\ee
where $Z'_{N}:=\Lambda_{-t}(Z_{N})$ is a $N$-dimensional subspace of $Y^{+}$ and where we used the fact that $\Lambda_{s}$ is an isomorphism for all $s\in\R$ and preserves $Y^{\pm}$. Now, since $h_t(0)=0$ and $\Lambda_t(0)=0$, we have $K_t(0)=0$. As a consequence
\[
 Y^-\oplus Z_N'\subset \left(Y^-\oplus Z_N'\right) + K_{-t}(Y^-\oplus Z_N),
\]
and hence, exploiting \eqref{inverseL} and the monotonicity of the genus,
\[
\gamma\left[ h_{t}(S^{+}_{r})\cap\cM_{N}\right]\geq\gamma\left[S^{+}_{r}\cap (Y^-\oplus Z'_{N})\right]\geq\gamma\left[S^{+}_{r}\cap Z'_{N}\right]=N,
\]
as $S^{+}_{r}\cap Z'_{N}\simeq\mathbb{S}^{N}$, namely, is homeomorphic to a $N$-dimensional sphere.
\end{proof}

Using Lemma \ref{genusintersezione} we can prove the existence of the Palais-Smale sequences at the min-max levels.

\begin{cor}
\label{genusminmax}
Let the assumptions of Lemma \ref{genusintersezione} be satisfied and define, for any $N\in\mathbb{N}$,
\be\label{minmax}
\alpha_{N}:=\inf_{X\in\cF_{N}}\sup_{\psi\in X}\cL(\psi),
\ee
with
\be\label{family}
\cF_{N}:=\left\{X\in\mathcal{A}:\gamma[h_{t}(S^{+}_{r})\cap X]\geq N,\,\forall t\geq0 \right\}.
\ee
Then, for every $N\in\mathbb{N}$, there exists a Palais-Smale sequence $\left( \psi_{n}\right)\subset Y$ at level $\alpha_{N}$, i.e. 
\[
 \begin{cases}
\displaystyle \cL(\psi_{n})\longrightarrow \alpha_{N}  \\[.2cm]
\displaystyle d\cL(\psi_{n})\xrightarrow{Y^{*}} 0,
\end{cases}
\qquad\mbox{as}\quad n\longrightarrow\infty.
\]
In addition, there results
\begin{gather}
\label{eq-levelorder}
\alpha_{N_{1}}\leq\alpha_{N_{2}}, \qquad \forall N_{1}<N_{2},\\[.2cm]
0<\rho\leq\alpha_{N}\leq\sup_{\cM_{N}}\cL<+\infty,\qquad\forall N\in\mathbb{N}.\nonumber
\end{gather}
\end{cor}

\begin{proof}
 The existence of a Palais-Smale sequence for $\cL$ at level $\alpha_{N}$ follows by standard deformation arguments, and then we only sketch the proof (see \cite{R-CBMS,S} for details).
 
 Preliminarily, we note that by Lemma \ref{genusintersezione} (and the definition of $\cM_N$) the classes $\cF_{N}$ are not empty, and hence that the levels $\alpha_N$ are well defined.
 
Now, suppose, by contradiction, that there is no Palais-Smale sequence at level $\alpha_{N}$. Then, since $\cL\in C^{1}$, there exist $\delta,\eps>0$ such that 
\be\label{nonzerogradient}
\Vert d\cL(\psi)\Vert\geq\delta,\qquad\forall\psi\in\{\alpha_{N}-2\eps<\cL<\alpha_{N}+2\eps\}.
\ee 
In addition, from \eqref{minmax} there exists $X_{\eps}\in\mathcal{F}_{N}$ such that 
\[
\sup_{\psi\in X_{\eps}}\cL(\psi)<\alpha_{N}+\eps,
\]
and hence, combining with \eqref{nonzerogradient}, we can see that there exists $T>0$ such that 
\[
\cL\left(h_{-T}(X_{\eps})\right)\subseteq\{\cL<\alpha_{N}-\eps\}.
\]
As a consequence, if one shows that $h_{-T}(X_{\eps})\in\mathcal{F}_{N}$, then obtains a contradiction.

First, observe that $h_{-T}(X_{\eps})\in\mathcal{A}$ as $h_{s}$ is odd, so that it suffices to prove that $$\gamma\left[h_{t}\left( S^{+}_{r}\right)\cap h_{-T}\left(X_{\eps}\right) \right]\geq N.$$ On the other hand, 
\[
h_{t}\left( S^{+}_{r}\right)\cap h_{-T}\left(X_{\eps}\right)=h_{-T}\left(h_{t+T}\left(S^{+}_{r}\right)\cap X_{\eps} \right),
\]
and then the monotonicity of the genus gives
\[
\gamma\left[h_{t}\left( S^{+}_{r}\right)\cap h_{-T}\left(X_{\eps}\right) \right]\geq \gamma\left[h_{t+T}\left( S^{+}_{r}\right)\cap X_{\eps} \right]\geq N.
\]
Therefore, $h_{-T}(X_{\eps})\in\mathcal{F}_{N}$ and this entails that
\[
\alpha_{N}\leq\sup_{\psi\in h_{-T}\left(X_{\eps}\right)}\cL(\psi)<\alpha_{N}-\eps,
\]
which is a contradiction.

Finally, the first line of \eqref{eq-levelorder} follows again by monotonicity of the genus, whereas the second one is a direct (up to some computations) consequence of Lemmas \ref{testlinking}, \ref{positivesphere} and \ref{genusintersezione}.
\end{proof}

\begin{remark}
It is easy to see that there are no non-trivial critical points for the action functional $\cL$ at levels $\alpha\leq0$. Indeed, let $\psi\in Y$ be such that $d\cL(\psi)=0$ and $\cL(\psi)=\alpha$. A simple computation gives
$$\left(\frac{1}{2}-\frac{1}{p} \right)\int_{\cK}\vert\psi\vert^{p}=\alpha, $$
which implies that $\alpha\geq0$. Suppose, now, that $\alpha=0$. Consequently, $\psi$ vanishes on the compact core $\cK$. Then, it follows that $\psi^{1}_e(\vv)=\psi^{2}_e(\vv)=0$, $\forall \vv,e\in\cK$, and thus, exploiting \eqref{eq-kirchtype1} and \eqref{explicithalflines}, that $\psi\equiv0$ on $\cG$.
\end{remark}
Now, before giving the proof of Theorem \ref{thm-bound}, we discuss the compactness properties of Palais-Smale sequences.

\begin{prop}
\label{prop-PSbounded}
For every $\alpha>0$, Palais-Smale sequences at level $\alpha$ are bounded in $Y$.
\end{prop}

\begin{proof}
Let $(\psi_{n})$ be a Palais-Smale sequence at level $\alpha>0$ and assume by contradiction that, up to subsequences, 
\[
\Vert\psi_{n}\Vert_{\omega}\longrightarrow\infty,\qquad\text{as}\quad n\rightarrow\infty.
\]
Simple computations show that, for $n$ large,
\[
\begin{split}
\left(\frac{1}{2}-\frac{1}{p} \right)\int_{\cK}\vert\psi_{n}\vert^{p}\dx=\cL(\psi_{n})-\frac{1}{2}\langle d\cL(\psi_{n})\vert\psi_{n}\rangle\leq C+\Vert\psi_{n}\Vert_{\omega}.
\end{split}
\]
and (recalling the definition of $P_\omega^{\pm}$ given by \eqref{eq-projomega})
\[
\left\vert\langle d\cL(\psi_{n})\vert P^{+}_{\omega}\psi_{n}\rangle \right\vert=\left\vert\int_{\cG}\langle P^{+}_{\omega}\psi_{n},(\cD-\omega)\psi_{n}\rangle\dx-\int_{\cK}\vert\psi_{n}\vert^{p-2}\langle\psi_{n},P^{+}_{\omega}\psi_{n}\rangle\dx \right\vert\leq\Vert\psi_{n}\Vert_{\omega}.
\]
As a consequence, using the H\"{o}lder inequality and \eqref{sobolev}, we get
\be\label{subquadratic}
\begin{split}
\left\vert \int_{\cG}\langle P^{+}_{\omega}\psi_{n},(\cD-\omega)\psi_{n}\rangle\dx\right\vert & \leq\Vert\psi_{n}\Vert_{\omega}+\int_{\cK}\vert\psi_{n}\vert^{p-1}\vert P^{+}_{\omega}\psi_{n}\vert\dx \\
&\leq\Vert\psi_{n}\Vert_{\omega}+\left(\int_{\cK}\vert\psi_{n}\vert^{p}\dx \right)^{\frac{p-1}{p}}\left(\int_{\cK}\vert P^{+}_{\omega}\psi_{n}\vert^{p}\dx \right)^{\frac{1}{p}}\\[.2cm]
& \leq C\left(1+\Vert\psi_{n}\Vert_{\omega}\right)^{1-\frac{1}{p}}\Vert\psi_{n}\Vert_{\omega}.
\end{split}
\ee
On the other hand, by the definition of $P_\omega^{\pm}$, one sees that
\[
\Vert P^{+}_{\omega}\psi_{n}\Vert^{2}_{\omega}=\int_{\cG}\langle P^{+}_{\omega}\psi_n,(\cD-\omega)P^{+}_{\omega}\psi_{n}\rangle\dx=\int_{\cG}\langle P^{+}_{\omega}\psi_n,(\cD-\omega)\psi_{n}\rangle\dx
\]
and, combining with \eqref{subquadratic},
\[
\Vert P^{+}_{\omega}\psi_{n}\Vert^{2}_{\omega}\leq C\left(1+\Vert\psi_{n}\Vert_{\omega}\right)^{1-\frac{1}{p}}\Vert\psi_{n}\Vert_{\omega}.
\]
Arguing as before, one also finds that
\[
\Vert P^{-}_{\omega}\psi_{n}\Vert^{2}_{\omega}\leq C\left(1+\Vert\psi_{n}\Vert_{\omega}\right)^{1-\frac{1}{p}}\Vert\psi_{n}\Vert_{\omega}
\]
and hence
\[
\Vert\psi_{n}\Vert^{2}_{\omega}\leq \left(C+\Vert\psi_{n}\Vert_{\omega}\right)^{1-\frac{1}{p}}\Vert\psi_{n}\Vert_{\omega},
\]
which is a contradiction if $\Vert\psi_{n}\Vert_{\omega}\to\infty$, since $1-\f{1}{p}\in(\f{1}{2},1)$ as $p>2$.
\end{proof}

\begin{lemma}
\label{compactness}
For every $\alpha>0$, Palais-Smale sequences at level $\alpha$ are pre-compact in $Y$.
\end{lemma}

\begin{proof}
Let $(\psi_{n})$ be a Palais-Smale sequence at level $\alpha>0$. From Proposition \ref{prop-PSbounded}, it is bounded and then, up to subsequences,
\begin{equation}\label{convergence}
 \begin{array}{ll}
\displaystyle \psi_{n}\rightharpoondown\psi, & \mbox{in}\quad Y,\\[.2cm]
\displaystyle \psi_{n}\longrightarrow\psi,   & \mbox{in}\quad L^{p}(\cK,\C^{2}).
\end{array}
\end{equation}
On the other hand, by definition
\be\label{zero}
\begin{split}
o(1) & =\langle d\cL(\psi_{n})\vert P^{+}_{\omega}(\psi_{n}-\psi)\rangle\\[.2cm]
     & = \int_{\cG}\langle P^{+}_{\omega}(\psi_{n}-\psi),(\cD-\omega)\psi_{n}\rangle\dx-\int_{\cK}\vert\psi_{n}\vert^{p-2}\langle\psi_{n},P^{+}_{\omega}(\psi_{n}-\psi)\rangle\dx,
\end{split}
\ee
and (again) by H\"{o}lder inequality and \eqref{convergence}
\[
\begin{split}
\left\vert\int_{\cK}\vert\psi_{n}\vert^{p-2}\langle\psi_{n},P^{+}_{\omega}(\psi_{n}-\psi)\rangle\dx\right\vert & \leq \int_{\cK}\vert\psi_{n}\vert^{p-1}\vert P^{+}_{\omega}\psi_{n}-\psi)\vert\dx\\[.2cm]
& \leq\Vert\psi_{n}\Vert^{p-1}_{L^{p}(\cK,\C^2)}\Vert P^{+}_{\omega}(\psi_{n}-\psi)\Vert_{L^{p}(\cK,\C^2)}=o(1).
\end{split}
\]
As a consequence, combining with \eqref{zero},
\be\label{eq-quadratic}
\int_{\cG}\langle P^{+}_{\omega}(\psi_{n}-\psi),(\cD-\omega)\psi_{n}\rangle\dx=o(1).
\ee
In addition, since $(\psi_{n}-\psi)\rightharpoondown 0$ in $Y$, we get
\[
\int_{\cG}\langle P^{+}_{\omega}(\psi_{n}-\psi),(\cD-\omega)\psi\rangle\dx=o(1)
\]
and, summing with \eqref{eq-quadratic}, there results
\[
\Vert P^{+}_{\omega}(\psi_{n}-\psi)\Vert^{2}_{\omega}=\int_{\cG}\langle(\cD-\omega)P^{+}_{\omega}(\psi_{n}-\psi),P^{+}_{\omega}(\psi_{n}-\psi)\rangle\dx=o(1).
\]
Since, analogously, one can prove that 
\[
\Vert P^{-}_{\omega}(\psi_{n}-\psi)\Vert^{2}_{\omega}=\int_{\cG}\langle(\cD-\omega)P^{-}_{\omega}(\psi_{n}-\psi),P^{-}_{\omega}(\psi_{n}-\psi)\rangle\dx=o(1),
\]
we obtain
\[
\Vert\psi_{n}-\psi\Vert^{2}_{\omega}=o(1),
\]
which concludes the proof.
\end{proof}
Finally, we have all the ingredients in order to prove Theorem \ref{thm-bound}.
\begin{proof}[Proof of Theorem \ref{thm-bound}]
By Corollary \ref{genusminmax}, for every $N\in\mathbb{N}$, there exists at least a Palais-Smale sequence at level $\alpha_N>0$ (defined by \eqref{minmax}) and, by Proposition \ref{compactness}, it converges to a critical point of $\cL$, which is via Proposition \ref{prop-boundcrit} a bound state of the NLDE.

Now, if the inequalities in \eqref{eq-levelorder} are strict, then one immediately obtains the claim. However, if $\alpha_j=\alpha_{j+1}=\dots=\alpha_{j+q}=\alpha$, for some $q\geq1$, then the claim follows by \cite[Proposition 10.8]{AM} as the properties of the genus imply the existence of infinitely many critical points at level $\alpha$.
\end{proof}

%%%%%%%%%%%%%%%%%%%%%%%%%%%%%%%%%%%%%%%%%%%%%%%%%%%%%%%%%%%%%%%%%%%%%%%%%%%%%%%%%%%%%%%%%%%%%%%%%%%%%%%%%%%%%%%%%%%%%%%%%%%%%%%%%%%%%%%%%%%%%%%%%%%%%%%%%%%%%%%%%%%%%%%%%%%%%%%%%%%%%%%%%%%%%%%%%%%%%%%%%%%%%%%%%%%%%%%%%%%%%%%%%%%%%%%%%%%%%%%%%%%%%%%%%%%%%%%%%%%%%%%%%%%%%%%%%%%%%%%%%%%%%%%%%%%%%%%%%%%%%%%%%%%%%%%%%%%%%%%%%%%%%%%%%%%%%%%%%%%%%%%%%%%%%%%%%%%%%%%%%%%%%%%%%%%%%%%%%%%%%%%%%%

\section{Nonrelativistic limit of solutions}
\label{sec-limit}

In this section we prove Theorem \ref{thm-limit}; namely, that there exists a wide class of (pairs of) sequences $(c_n)$, $(\omega_n)$ for which the nonrelativistic limit holds. More precisely, we show that with such a choice of parameters the bound states of the NLDE converge, as $c_{n}\rightarrow+\infty$, to the bound states of the NLSE with $\alpha=2m$.

The strategy that we use is the one developed by M.J. Esteban and  E. S\'er\'e in \cite{ES-ANIHPA} for the case of Dirac-Fock equations. However, the differences between both the equations and the frameworks discussed call for some relevant modifications. In particular, while in \cite{ES-ANIHPA} one of the main point is the estimate of the sequence of the Lagrange multipliers of bound states with $L^2$-norm fixed, here the major point (since there is no constraint) is to prove that the limit is non-trivial. Moreover, we also have to distinguish different cases according to the exponent $p\in(2,6)$ of the nonlinearity.

\medskip
Preliminarily, note that, since here the role of the (sequence of the) speed of light is central, we cannot set any more $c=1$. As a consequence, all the previous results has to be meant with $m$ replaced by $mc_n^2$ (and $\omega$ replaced by $\omega_n$).
In addition, we denote by $\cD_n$ the Dirac operator with $c=c_n$ and with $\cL_n$ the action functional with $\cD=\cD_n$ and $\omega=\omega_n$. There are clearly many other quantities which actually depend on the index $n$ (such as, for instance, the form domain $Y$, $Z_N$, $\dots$), but since such a dependence is not crucial we omit it for the sake of simplicity. In addition, in the following, we will always make the assumptions \eqref{cdiverge},\eqref{eq-smaller} and \eqref{parameterconvergence} on the parameters $(c_n)$, $(\omega_n)$. In particular, those assumptions immediately imply that
\begin{equation}
 \label{frequency}
 0<C_1\leq mc^{2}_{n}-\omega_{n}\leq C_2.
\end{equation}

Now, from Theorem \ref{thm-bound}, for every fixed $N\in\mathbb{N}$, there exist at least a pair of bound states of frequency $\omega_n$ at level $\alpha^{n}_{N}$ of the NLDE at speed of light $c_n$. Hence, we denote throughout by $(\psi_n)$ a sequence of bound states corresponding to those values of parameters. Since all the following results hold for every fixed $N\in\mathbb{N}$, the dependence on $N$ is understood in the sequel (unless stated otherwise).

%%%%%%%%%%%%%%%%%%%%%%%%%%%%%%%%%%%%%%%%%%%%%%%%%%%%%%%%%%%%%%%%%%%%%%%%%%%%%%%%%%%%%%%

\subsection{$H^1$-boundedness of the sequence of the bound states}

The first step is to prove that the sequence $(\psi_n)$ defined above is bounded in $L^{p}(\cK,\C^2)$.

\begin{lemma}
\label{lpbound}
Under the assumptions \eqref{cdiverge}, \eqref{eq-smaller} and \eqref{parameterconvergence}, the sequence $(\psi_n)$ is bounded in $L^{p}(\cK,\C^2)$ (uniformly with respect to $n$), as well as the associated minimax levels $(\alpha^{n}_{N})$.
\end{lemma}

\begin{proof}
First, recalling \eqref{family} and following the notation of the proof of Lemma \ref{testlinking}, one sees that
\[
\alpha^{n}_{N}=\inf_{X\in\cF_{N}}\sup_{\psi\in X}\cL_{n}(\psi)\leq\sup_{Y^{-}\oplus Z_{N}}\cL_{n}.
\]
In addition, following again the proof of Lemma \ref{testlinking}, given an orthonormal basis $\eta^{+}_{j}, j=1,...,N$, of $Z_{N}$, every spinor $\psi\in Y^{-}\oplus Z_{N}$ can be decomposed as
\[
\psi=\varphi^{\perp}+\sum^{N}_{j=1}\lambda_{j}\eta_{j},\qquad \lambda_{1},\dots,\lambda_N\in\C,
\]
with $\varphi^{\perp}\in Y^{-}$ orthogonal to $\zeta:=\sum^{N}_{j=1}\lambda_{j}\eta_{j}\in V$. Arguing as in \eqref{holder}--\eqref{eq-quasifinal} we get
\be
\label{eq-estlev1}
\cL_{n}(\psi)\leq\frac{1}{2}\int_{\cG}\langle\zeta,(\cD_n-\omega_n)\zeta\rangle\dx-C\left( \int_{\cG}\vert\zeta\vert^{2}\dx\right)^{\frac{p}{2}}.
\ee
On the other hand, exploiting \eqref{partepositiva} and \eqref{frequency}, there results
\be
\label{eq-estlev2}
\int_{\cG}\langle\zeta,(\cD_n-\omega_n)\zeta\rangle\dx=(mc^{2}_{n}-\omega_{n})\int_{\cG}\vert\zeta\vert^{2}\dx\leq C \int_{\cG}\vert\zeta\vert^{2}\dx.
\ee
Hence, combining \eqref{eq-estlev1} and \eqref{eq-estlev2},
\[
\cL_{n}(\psi) \leq C \int_\cG|\zeta|^2\dx\left[1-\bigg(\int_\cG|\zeta|^2\dx\bigg)^{\f{p-2}{2}}\right]
\]
and thus, since $V$ does not depend on $n$ and since $p>2$
\[
\alpha^{n}_{N}\leq\max_{Y^{-}\oplus Z_{N}}\cL_{n}\leq C<+\infty,\qquad \forall n\in\mathbb{N}.
\]
Finally, as $\psi_{n}$ is a critical point of the action functional,
\[
\alpha^{n}_{N}=\cL_{n}(\psi_{n})-\frac{1}{2}\langle d\cL_{n}(\psi_{n}),\psi_{n}\rangle=\left(\frac{1}{2}-\frac{1}{p} \right)\int_{\cK}\vert\psi_{n}\vert^{p},
\]
which concludes the proof.
\end{proof}

We can now prove that boundedness on $L^{p}(\cK,\C^2)$ entails boundedness on $L^{2}(\cG,\C^2)$.

\begin{lemma}
\label{l2bound}
Under the assumptions \eqref{cdiverge}, \eqref{eq-smaller} and \eqref{parameterconvergence}, the sequence $(\psi_n)$ is bounded in $L^{2}(\cG,\C^2)$ (uniformly with respect to $n$).
\end{lemma}

\begin{proof}
For the sake of simplicity, denote by $\psi^{\pm}$ the projections of the spinor $\psi\in Y$ given by \eqref{eq-projomega} (with $\omega=\omega_n$). As the spectrum of the operator $\cD_{n}-\omega_{n}$ is 
\be
\label{eq-newsp}
\sigma(\cD_{n}-\omega_{n})=(-\infty,-mc^{2}_{n}-\omega_{n}]\cup[mc^{2}_{n}-\omega_{n},+\infty)
\ee
and $\psi_n$ satisfies \eqref{eq-NLDbound} (with $c=c_n$ and $\omega=\omega_n$),
H\"{o}lder inequality yields 
\[
\begin{split}
0\leq\int_{\cG}\langle\psi^{+}_{n},(\cD_{n}-\omega_{n})\psi_{n}^+\rangle\dx&=\int_{\cG}\langle\psi^{+}_{n},(\cD_{n}-\omega_{n})\psi_{n}\rangle\dx\leq\int_{\cK}\vert\psi_{n}\vert^{p-1}\vert\psi^{+}_{n}\vert\dx\\[.2cm]
&\leq\left(\int_{\cK}\vert\psi_{n}\vert^{p}\dx \right)^{\frac{p-1}{p}}\left(\int_{\cK}\vert\psi^{+}_{n}\vert^{p}\dx \right)^{\frac{1}{p}}\leq C\int_{\cK}\vert\psi_{n}\vert^{p}\dx
\end{split}
\]
for some $C>0$, where in the last inequality we used the fact that the decomposition 
\[ 
Y=Y^{+}_{\omega_n}\oplus Y^{-}_{\omega_n}
\]
induces an analogous decomposition on $L^p(\cK)$, that is
\[
\Vert \psi^{\pm}\Vert_{L^{p}(\cK,\C^2)}\leq \Vert\psi\Vert_{L^{p}(\cK,\C^2)},\quad\forall \psi\in Y.
\]
Moreover, using \eqref{eq-newsp} one can prove that
\[
\int_{\cG}\langle\psi^{+}_{n},(\cD_{n}-\omega_{n})\psi^{+}_{n}\rangle\dx\geq(mc^{2}_{n}-\omega_{n})\int_{\cG}\vert\psi^{+}_{n}\vert^{2}\dx.
\]
Then, combining the above observations with Lemma \ref{lpbound} and \eqref{frequency}, there results 
\[
\Vert\psi^{+}_{n}\Vert_{L^{2}(\cG,\C^2)}^2 \leq M<\infty.
\]
An analogous argument gives
\[
(mc^{2}_{n}+\omega_{n})\Vert\psi^{-}_{n}\Vert_{L^{2}(\cG,\C^2)}^2\leq M<\infty
\]
and then
\[
\Vert\psi^{-}_{n}\Vert_{L^{2}(\cG,\C^2)}=\mathcal{O}\left(\frac{1}{c_{n}}\right),\qquad\text{as}\quad n\rightarrow+\infty,
\]
which concludes the proof.
\end{proof}

Finally, we can deduce boundedness in $H^1(\cG,\C^2)$. Preliminarily, we recall two Gagliardo-Nirenberg inequalities for spinors that can be easily deduced from those on functions (see e.g. \cite[Proposition 2.6]{ST-NA}). For every $p\geq2$, there exists $C_p>0$ such that
\be
\label{eq-gnp}
\Vert\psi\Vert^{p}_{L^{p}(\cG,\C^2)}\leq C_{p}\Vert\psi\Vert^{\frac{p}{2}+1}_{L^{2}(\cG,\C^2)}\Vert\psi'\Vert^{\frac{p}{2}-1}_{L^{2}(\cG,\C^2)},\qquad\forall \psi\in H^1(\cG,\C^2).
\ee
Moreover, there exists $C_\infty>0$ such that
\be\label{gninfinito}
\Vert\psi\Vert_{L^{\infty}(\cG,\C^2)}\leq C_{\infty}\Vert\psi\Vert^{\frac{1}{2}}_{L^{2}(\cG,\C^2)}\Vert\psi'\Vert^{\frac{1}{2}}_{L^{2}(\cG,\C^2)},\qquad\forall \psi\in H^1(\cG,\C^2).
\ee

\begin{lemma}\label{h1bound}
Let $p\in(2,6)$. Under the assumptions \eqref{cdiverge}, \eqref{eq-smaller} and \eqref{parameterconvergence}, the sequence $(\psi_n)$ is bounded in $H^{1}(\cG,\C^2)$ (uniformly with respect to $n$).
\end{lemma}

\begin{proof}
First, recall that, since $\psi_n$ are bound states they satisfy (edge by edge)
\be\label{equazionen}
\cD_{n}\psi_{n}=\omega_{n}\psi_{n}+\chi_{_\cK}\vert\psi_{n}\vert^{p-2}\psi_{n}.
\ee
The $L^{2}(\cG,\C^2)$-norm squared of the right-hand side of \eqref{equazionen} reads
\be\label{destraquadrato}
\Vert\omega_{n}\psi_{n}+\kappa\vert\psi_{n}\vert^{p-2}\psi_{n}\Vert^{2}_{L^{2}(\cG,\C^2)}=\omega^{2}_{n}\Vert\psi_{n}\Vert^{2}_{L^{2}(\cG,\C^2)}+\int_{\cK}\vert\psi_{n}\vert^{2(p-1)}\dx+2\omega_{n}\int_{\cK}\vert\psi_{n}\vert^{p}\dx.
\ee
Let us estimate the last two integrals. Using \eqref{gninfinito}, Lemma \ref{lpbound} and Lemma \ref{l2bound}, we get
\be\label{primastima}
\begin{split}
\int_{\cK}\vert\psi_{n}\vert^{2(p-1)}&=\int_{\cK}\vert\psi_{n}\vert^{p+(p-2)}\dx\leq\Vert\psi_{n}\Vert^{p-2}_{L^{\infty}(\cG,\C^2)}\int_{\cK}\vert\psi_{n}\vert^{p}\dx\\[.2cm]
&\leq C^{p-2}_{\infty}\Vert\psi_{n}\Vert^{\frac{p}{2}+1}_{L^{2}(\cG,\C^2)}\Vert\psi_{n}'\Vert^{\frac{p}{2}-1}_{L^{2}(\cG,\C^2)}\leq C\Vert\psi_{n}'\Vert^{\frac{p}{2}-1}_{L^{2}(\cG,\C^2)}.
\end{split}
\ee
On the other hand, by \eqref{eq-gnp} and Lemma \ref{l2bound}
\be\label{secondastima}
\int_{\cK}\vert\psi_{n}\vert^{p}\dx\leq\int_{\cG}\vert\psi_{n}\vert^{p}\dx\leq C_{p}\Vert\psi_{n}\Vert^{\frac{p}{2}+1}_{L^{2}(\cG,\C^2)}\Vert\psi_{n}'\Vert^{\frac{p}{2}-1}_{L^{2}(\cG,\C^2)}\leq C\Vert\psi_{n}'\Vert^{\frac{p}{2}-1}_{L^{2}(\cG,\C^2)}.
\ee
Since an easy computation shows that
\be\label{quadrato}
\Vert\cD_{n}\psi_{n}\Vert^{2}_{L^{2}(\cG,\C^2)}=c^{2}_{n}\Vert\psi_{n}'\Vert^{2}_{L^{2}(\cG,\C^2)}+m^{2}c^{4}_{n}\Vert\psi_{n}\Vert^{2}_{L^{2}(\cG,\C^2)},
\ee
combining \eqref{destraquadrato}, \eqref{primastima}, \eqref{secondastima} and \eqref{quadrato}, we obtain that
\[
c^{2}_{n}\Vert\psi_{n}'\Vert^{2}_{L^{2}(\cG,\C^2)}+m^{2}c^{4}_{n}\Vert\psi_{n}\Vert^{2}_{L^{2}(\cG,\C^2)}\leq\omega^{2}_{n}\Vert\psi_{n}\Vert^{2}_{L^{2}(\cG,\C^2)}+ C(1+\omega_{n})\Vert\psi_{n}'\Vert^{\frac{p}{2}-1}_{L^{2}(\cG,\C^2)},
\]
so that, from a repeated use of \eqref{cdiverge} and \eqref{eq-smaller},
\[
 \Vert\psi_{n}'\Vert^{\frac{6-p}{2}}_{L^{2}(\cG,\C^2)}\leq C m.
\]
Hence, the claim follows by the assumption $p<6$.
\end{proof}
\begin{remark}
The above results also hold if \eqref{parameterconvergence} is replaced by the weaker assumption \eqref{frequency}.
\end{remark}

%%%%%%%%%%%%%%%%%%%%%%%%%%%%%%%%%%%%%%%%%%%%%%%%%%%%%%%%%%%%%%%%%%%%%%%%%%%%%%%%%%%%%%%

\subsection{Passage to the limit}

The last step consists in proving that the first components of the sequence of bound states $(\psi_n)$ converge to a bound state of the NLSE, while the second components converge to zero.

For the sake of simplicity we assume throughout that the parameters $p$ and $\lambda$ are fixed and fulfill
\[
 p\in(2,6)\qquad \text{and}\qquad \lambda<0.
\]
In addition, we set
\[
u_n:=\psi_{n}^1\qquad\text{and}\qquad v_n:=\psi_n^2,\qquad \forall n\in\mathbb{N},
\]
and, given the two sequences $(c_n)$ and $(\omega_n)$ introduced in the previous section (which satisfy \eqref{cdiverge}, \eqref{eq-smaller},\eqref{parameterconvergence} and \eqref{frequency}), we define
\[
 a_n:=(mc_n^2-\omega_n)b_n\qquad\text{and}\qquad b_n:=\f{mc_n^2+\omega_n}{c_n^2},\qquad \forall n\in\mathbb{N}.
\]
Clearly, \eqref{frequency} implies that
\begin{equation}
\label{eq-bn}
 b_n\longrightarrow2m,\qquad\text{as}\quad n\to\infty
\end{equation}
while \eqref{parameterconvergence} gives
\begin{equation}
 \label{eq-an}
 a_n\longrightarrow-\lambda,\qquad\text{as}\quad n\to\infty.
\end{equation}

We also recall that a function $w:\cG\to\C$ is a bound state of the NLSE with fixed frequency $\lambda$ and $\alpha=2m$ if and only if it is a critical point of the $C^2$ functional $J:H\to\R$ defined by
\[
 J(w):=\f{1}{2}\int_\cG|w'|^2\dx-\frac{2m}{p}\int_\cK|w|^p\dx-\f{\lambda}{2}\int_\cG|w|^2\dx,
\]
where
\[
 H:=\{w\in H^1(\cG):\text{\eqref{eq-kirch1} holds}\}
\]
with the norm induced by $H^1(\cG)$ (this can be easily proved arguing as in \cite[Proposition 3.3]{AST-CVPDE}). It is also worth mentioning that a Palais-Smale sequence for $J$ is a sequence $(w_n)\subset H$ such that $dJ(w_n)\to0$ in $H^*$, namely
\[
 \sup_{\|\phi\|_{H}\leq 1}\langle dJ(w_n)|\phi\rangle\to0,\qquad\text{as}\quad n\to\infty.
\]
Furthermore, there holds the following property.

\begin{lemma}
\label{lem-equiv}
 Let $(w_n)$ be a bounded sequence in $H$ and, for every $n$, define the linear functional $A_n(w_n):H\to\R$
 \[
  \langle A_n(w_n)|\phi\rangle:=\int_\cG w_n'\overline{\phi}'\dx-b_n\int_\cK|w_n|^{p-2}w_n\overline{\phi}\dx+a_n\int_\cG w_n\overline{\phi}\dx.
 \]
 Then, $(w_n)$ is a Palais-Smale sequence for $J$ if and only if
 \begin{equation}
 \label{eq-psaux}
  \sup_{\|\phi\|_{H}\leq 1}\langle A_n(w_n)|\phi\rangle\to0,\qquad\text{as}\quad n\to\infty.
 \end{equation}
\end{lemma}

\begin{proof}
 The proof is trivial noting that
 \[
  \langle A_n(w_n)-dJ(w_n)|\phi\rangle=-(b_n-m)\int_\cK|w_n|^{p-2}w_n\overline{\phi}\dx+(a_n+\lambda)\int_\cG w_n\overline{\phi}\dx
 \]
and exploiting \eqref{eq-bn}, \eqref{eq-an} and the fact that $(w_n)$ is bounded in $H$.
\end{proof}

The strategy to prove Theorem \ref{thm-limit} is the following:
\begin{itemize}
 \item[(i)] prove that the sequence $(v_n)$ converges to $0$ in $H^1(\cG)$;\\[-.35cm]
 \item[(ii)] prove that the sequence $(u_n)$ is bounded away from zero in $H^1(\cG)$;\\[-.35cm]
 \item[(iii)] prove that the sequence $(u_n)$ satisfies \eqref{eq-psaux}, as by Lemmas \ref{h1bound} and \ref{lem-equiv}, this entails that it is a Palais-Smale sequence for $J$;\\[-.35cm]
 \item[(iv)] prove that the sequence $(u_n)$ converges (up to subsequences) in $H$ to a function $u$, which is then a bound state of the NLSE with frequency $\lambda<0$.
\end{itemize}

We observe that we always tacitly use in the following the fact that, since each $\psi_n$ is a bound state of the NLDE, then $u_n\in H$, whereas $v_n\not\in H$, but satisfies \eqref{eq-kirchtype2}. In addition, we highlight that, in the sequel, we often use a ``formal'' commutation between the differential operator $(\cdot)'$ and $\chi_{_\cK}$. Clearly, this is just a compact notation (which avoids tedious edge by edge computations) that simply recalls the different form of the NLDE on the bounded edges due to the presence of the localized nonlinearity.

\medskip
As a first step, we prove item (i). As a byproduct of the proof, we also find an estimate of the speed of convergence of $(v_n)$.

\begin{lemma}
\label{lem-vnzero}
 The sequence $(v_n)$ converges to $0$ in $H^1(\cG)$ as $n\to\infty$. More precisely, there holds 
 \be\label{h10}
 \Vert v_{n}\Vert_{H^{1}(\cG)}=\cO\left(\frac{1}{c_{n}} \right),\qquad\text{as}\quad n\to\infty.
 \ee
\end{lemma}

\begin{proof}
 As $(\psi_n)$ is a bound state of the NLDE, rewriting the equation in terms of its components,
 \begin{gather}
\label{eqcomponenti1}
\displaystyle -\imath c_{n}v_{n}'+(mc^{2}_{n}-\omega_{n})u_{n}=\chi_{_\cK}(\vert u_{n}\vert^{2}+\vert v_{n}\vert^{2})^{\frac{p-2}{2}}u_{n}, \\[.3cm]
\label{eqcomponenti2}
\displaystyle -\imath c_{n}u_{n}'-(mc^{2}_{n}+\omega_{n})v_{n}=\chi_{_\cK}(\vert u_{n}\vert^{2}+\vert v_{n}\vert^{2})^{\frac{p-2}{2}}v_{n}.
\end{gather}
Dividing \eqref{eqcomponenti1} by $c_{n}$ and using \eqref{frequency} and Lemma \ref{h1bound}, we have
\be\label{gradiente0}
\Vert v_{n}'\Vert_{L^{2}(\cG)}=\cO\left(\frac{1}{c_{n}}\right).
\ee
On the other hand, dividing \eqref{eqcomponenti2} by $c^{2}_{n}$ and using again Lemma \ref{h1bound}, there results
\[
\left\Vert-\frac{\imath}{c_{n}} u_{n}'-\frac{(mc^{2}_{n}+\omega_{n})}{c^{2}_{n}}v_{n}\right\Vert_{L^{2}(\cG)}=\cO\left( \frac{1}{c^{2}_{n}}\right)
\]
and hence
\[
\frac{(mc^{2}_{n}+\omega_{n})}{c^{2}_{n}}\Vert v_{n}\Vert_{L^{2}(\cG)}\leq \left\Vert-\frac{\imath}{c_{n}}u_{n}'-\frac{(mc^{2}_{n}+\omega_{n})}{c^{2}_{n}}v_{n}\right\Vert_{L^{2}(\cG)}+\frac{1}{c_{n}}\left\Vert u_{n}'\right\Vert_{L^{2}(\cG)}= \cO\left(\frac{1}{c_{n}} \right).
\]
Finally, combining with \eqref{gradiente0}, one obtains \eqref{h10}.
\end{proof}

Item (ii) requires some further effort.

\begin{lemma}
\label{lem-nonzero}
There exists $\mu>0$ such that
\be\label{lowerboundH1}
\inf_{n\in\mathbb{N}}\Vert u_{n}\Vert_{H^{1}(\cG)}\geq\mu>0.
\ee
\end{lemma}
\begin{proof} 
 Assume, by contradiction, that \eqref{lowerboundH1} does not hold, namely that, up to subsequences,
 \be\label{H1zero}
 \lim_{n\rightarrow\infty}\Vert u_{n}\Vert_{H^{1}(\cG)}=0.
 \ee
 
 Dividing by $c_n$ and rearranging terms, \eqref{eqcomponenti1} yields
\be
\label{eq-comp1rew}
-\imath v'_{n}=\frac{1}{c_{n}}\left[\chi_{_\cK}\left(\vert u_{n}\vert^{2}+\vert v_{n}\vert^{2} \right)^{\frac{p-2}{2}}+(\omega_{n}-mc^{2}_{n}) \right]u_{n},
\ee
and then, using \eqref{frequency}, we find
\be\label{vouL2}
\int_{\cG}\vert v'_{n}\vert^{2}\dx\lesssim\frac{1}{c^{2}_{n}}\int_{\cG}\vert u_{n}\vert^{2}\dx.
\ee
Moreover,  \eqref{eqcomponenti2} can be rewritten as
\be
\label{eq-comp2rew}
v_{n}\left(1+\chi_{_\cK}\frac{\left(\vert u_{n}\vert^{2}+\vert v_{n}\vert^{2}\right)^{\frac{p-2}{2}}}{mc^{2}_{n}+\omega_{n}} \right)=- \frac{\imath c_{n}}{mc^{2}_{n}+\omega_{n}}u'_{n}
\ee
and, since (again by \eqref{frequency}) $(mc^{2}_{n}+\omega_{n})\sim c^{2}_{n}$, there results
\[
\int_{\cG}\vert v_{n}\vert^{2}\dx\lesssim\frac{1}{c^{2}_{n}}\int_{\cG}\vert u'_{n}\vert^{2}\dx,
\]
so that, combining with \eqref{vouL2},
\be\label{vou}
\Vert v_{n}\Vert_{H^{1}(\cG)}\lesssim\frac{1}{c_{n}}\Vert u_{n}\Vert_{H^{1}(\cG)}.
\ee
Note that \eqref{eq-comp2rew} also shows that $u_n$ is of class $C^1$ on each edge.
 
 Now, plugging \eqref{eq-comp2rew} into \eqref{eq-comp1rew}, one obtains
 \be\label{PSremainder}
 -u_{n}''+a_nu_{n}=-\frac{\imath\chi_{_\cK}}{c_{n}}\left[\left(\vert u_{n}\vert^{2}+\vert v_{n}\vert^{2}\right)^{\f{p-2}{2}}v_{n} \right]'+\chi_{_\cK}b_n\left[\left(\vert u_{n}\vert^{2}+\vert v_{n}\vert^{2}\right)^{\f{p-2}{2}}u_{n}\right].
\ee
Clearly, \eqref{PSremainder} is to be meant in a distributional sense. However, observing that it can be written as
\[
 -\left[u_{n}'-\frac{\imath\chi_{_\cK}}{c_{n}}\left(\vert u_{n}\vert^{2}+\vert v_{n}\vert^{2}\right)^{\f{p-2}{2}}v_{n} \right]'=-a_nu_{n}+\chi_{_\cK}b_n\left[\left(\vert u_{n}\vert^{2}+\vert v_{n}\vert^{2}\right)^{\f{p-2}{2}}u_{n}\right],
\]
and that consequently the l.h.s. belongs to $L^2(\cG)$ and is continuous edge by edge (recalling also that $u_n$ is of class $C^1$ edge by edge by \eqref{eqcomponenti2}), the following multiplications by $\overline{u}_n$ and integrations (by parts) can be proved to be rigorous in the Lebesgue sense.

Therefore, multiplying \eqref{PSremainder} by $\overline{u}_{n}$ and integrating (by parts) over $\cG$, at the l.h.s. we obtain
\[
\int_{\cG}\vert u'_{n}\vert^{2}\dx+\underbrace{\sum_{\vv\in\cK}\left(\sum_{e\succ\vv}\overline{u}_{n,e}(\vv)\frac{du_{n,e}}{dx_e}(\vv) \right)}_{\mbox{boundary terms}}+a_n\int_{\cG}\vert u_{n}\vert^{2}\dx
\]
where we denote by $u_{n,e}$ (and $v_{n,e}$) the restriction of $u_{n}$ (and $v_{n}$) to the edge (represented by) $I_{e}$, and $\frac{d}{dx_e}$ is to be meant as in Definition \ref{defi-NLSE}. Using \eqref{eqcomponenti2} and the fact that $u_n$ is of class $C^1$ (edge by edge), we find that
\begin{multline*}
\sum_{e\succ\vv}\overline{u}_{n,e}(\vv)\frac{du_{n,e}}{dx_e}(\vv)=\\[.2cm]
=-\frac{\imath}{c_{n}}\sum_{e\succ\vv}\overline{u}_{n,e}(\vv)\left((mc^{2}_{n}+\omega_{n})v_{n,e}(\vv)_{\pm}+\left(\vert u_{n,e}(\vv)\vert^{2}+\vert v_{n,e}(\vv)\vert^{2}\right)^{\frac{p-2}{2}}v_{n,e}(\vv)_{\pm} \right)\\[.2cm]
=\underbrace{-\frac{\imath}{c_{n}}\sum_{e\succ\vv}\overline{u}_{n,e}(\vv)(mc^{2}_{n}+\omega_{n})v_{n,e}(\vv)_{\pm}}_{=:A}-\frac{\imath}{c_{n}}\sum_{e\succ\vv}\overline{u}_{n,e}(\vv)v_{n,e}(\vv)_{\pm}\left(\vert u_{n,e}(\vv)\vert^{2}+\vert v_{n,e}(\vv)\vert^{2}\right)^{\frac{p-2}{2}}
\end{multline*}
($v_{n,e}(\vv)_{\pm}$ meant as in Definition \ref{defi-dirac}). Moreover, as $u_{n}$ and $v_{n}$ satisfy the vertex conditions \eqref{eq-kirchtype1} and \eqref{eq-kirchtype2} (respectively), one has
\[
A=-\frac{\imath(mc^{2}_{n}+\omega_{n})}{c_{n}}\overline{u}_{n}(\vv)\sum_{e\succ\vv}v_{n,e}(\vv)_{\pm}=0,
\]
while, for any $\vv\in\cK$ and $e\succ\vv$, there results
\[
\begin{split}
\left\vert\overline{u}_{n,e}(\vv)v_{n,e}(\vv)_{\pm}\left(\vert u_{n,e}(\vv)\vert^{2}+\vert v_{n,e}(\vv)\vert^{2}\right)^{\frac{p-2}{2}}\right\vert&\leq \Vert\vert u_{n}\vert^{2}+\vert v_{n}\vert^{2}\Vert^{\frac{p-2}{2}}_{L^{\infty}(\cG)}\Vert u_{n} \Vert_{L^{\infty}(\cG)}\Vert v_{n} \Vert_{L^{\infty}(\cG)}\\[.2cm]
& \lesssim\Vert u_{n} \Vert_{H^{1}(\cG)}\Vert v_{n} \Vert_{H^{1}(\cG)}=o\left(\Vert u_{n}\Vert^{2}_{H^{1}(\cG)}\right)
\end{split}
\]
(where we used Lemma \ref{h1bound}, \eqref{vou} and Sobolev embeddings). As a consequence (since the number of the edges and the vertices is finite)
\[
\sum_{\vv\in\cK}\left(\sum_{e\succ\vv}\overline{u}_{n,e}(\vv)\frac{du_{n,e}}{dx_e}(\vv) \right)=o\left(\Vert u_{n}\Vert^{2}_{H^{1}(\cG)}\right),
\]
so that (recalling \eqref{eq-an})
\be\label{sobolevlowerbound}
\int_{\cG}\vert u'_{n}\vert^{2}\dx+a_n\int_{\cG}\vert u_{n}\vert^{2}\dx\gtrsim\Vert u_{n}\Vert^{2}_{H^{1}(\cG)}.
\ee

Let us focus on the r.h.s. of \eqref{PSremainder}. After multiplication times $\overline{u}_n$ and integration over $\cG$ we have
\[
 -\frac{\imath}{c_{n}}\int_\cK\overline{u}_n\left[\left(\vert u_{n}\vert^{2}+\vert v_{n}\vert^{2}\right)^{\f{p-2}{2}}v_{n} \right]'\dx+b_n\int_{\cK}\left( \vert u_{n}\vert^{2}+\vert v_{n}\vert^{2}\right)^{\frac{p-2}{2}}\vert u_{n}\vert^{2}\dx.
\]
The latter term can be easily estimated using the H\"{o}lder inequality and \eqref{eq-bn}, \eqref{H1zero} and \eqref{vou}, i.e.
\begin{multline*}
b_n\int_{\cK}\left( \vert u_{n}\vert^{2}+\vert v_{n}\vert^{2}\right)^{\frac{p-2}{2}}\vert u_{n}\vert^{2}\dx\lesssim\left(\Vert u_{n}\Vert^{p-2}_{L^{\infty}(\cG)}+\Vert v_{n}\Vert^{p-2}_{L^{\infty}(\cG)} \right)\int_{\cG}\vert u_{n}\vert^{2}\dx\\[.2cm]
\lesssim \left(\Vert u_{n}\Vert^{p-2}_{H^{1}(\cG)}+\Vert v_{n}\Vert^{p-2}_{H^{1}(\cG)} \right)\Vert u_{n}\Vert^{2}_{H^{1}(\cG)}=o\left(\Vert u_{n}\Vert^{2}_{H^{1}} \right).
\end{multline*}
On the contrary, the former one requires some further efforts. Clearly,
\begin{align}
 \label{derivativeterm}
\frac{1}{c_{n}}\int_{\cK}\overline{u}_n\left[ \left(\vert u_{n}\vert^{2}+\vert v_{n}\vert^{2}\right)^{\frac{p-2}{2}}v_{n}\right]'\dx = & \, \underbrace{\frac{1}{c_{n}}\int_{\cK}v'_{n}\overline{u}_{n}\left(\vert u_{n}\vert^{2}+\vert v_{n}\vert^{2} \right)^{\frac{p-2}{2}}\dx}_{=:I_1}+\nonumber\\[.2cm]
& \, +\underbrace{\frac{1}{c_{n}}\int_{\cK}\overline{u}_{n}v_{n}\left(\vert u_{n}\vert^{2}+\vert v_{n}\vert^{2} \right)^{\frac{p-4}{2}}\left( \overline{u}_{n}u'_{n}+\overline{v}_{n}v'_{n}\right)\dx}_{I_2}.
\end{align}
Using \eqref{eq-comp1rew} and again Lemma \ref{h1bound}, we immediately find that
\[
\begin{split}
\left\vert I_1\right\vert\lesssim\frac{1}{c_{n}^2}\Vert u_{n}\Vert^{2}_{H^{1}(\cG)}=o\left(\Vert u_{n}\Vert^{2}_{H^{1}(\cG)}\right).
\end{split}
\]
It is, then, left to estimate $I_2$. We distinguish two cases.

\textbf{Estimate for $I_2$, case $p\in(2,4)$}: as $p-4<0$ there holds
\[
\left(\vert u_{n}\vert^{2}+\vert v_{n}\vert^{2} \right)^{\frac{p-4}{2}}\leq 2^{\frac{p-4}{2}}\left(\vert u_{n}\vert\vert v_{n}\vert \right)^{\frac{p-4}{2}}.
\]
As a consequence
\[
\vert I_2\vert\lesssim\f{1}{c_n}\int_{\cK}\vert u_{n}\vert^{\frac{p}{2}}\vert v_{n}\vert^{\frac{p-2}{2}}\vert u'_{n}\vert\dx +\f{1}{c_n}\int_{\cK}\vert u_{n}\vert^{\frac{p-2}{2}}\vert v_{n}\vert^{\frac{p}{2}}\vert v'_{n}\vert\dx=:I_{2,1}+I_{2,2}.
\]
Moreover,
\begin{multline*}
I_{2,1}\lesssim\f{1}{c_n}\Vert v_{n}\Vert^{\frac{p-2}{2}}_{L^{\infty}(\cG)}\int_{\cK}\vert u_{n}\vert^{\frac{p}{2}}\vert u'_{n}\vert\dx\\[.2cm]
\leq\f{1}{c_n}\Vert v_{n}\Vert^{\frac{p-2}{2}}_{L^{\infty}(\cG)}\Vert u_{n}\Vert^{\frac{p}{2}}_{L^{p}(\cG)}\Vert u'_{n}\Vert_{L^{2}(\cG)}\lesssim\f{1}{c_n}\Vert v_{n}\Vert^{\frac{p-2}{2}}_{L^{\infty}(\cG)}\Vert u_{n}\Vert^{\frac{p}{2}+1}_{H^{1}(\cG)},
\end{multline*}
whereas
\[
I_{2,2}\lesssim \f{1}{c_n}\Vert u_{n}\Vert^{\frac{p-2}{2}}_{L^{\infty}}\Vert v_{n}\Vert^{\frac{p}{2}+1}_{H^{1}(\cG)}\lesssim\f{1}{c_n}\Vert u_{n}\Vert^{\frac{p-2}{2}}_{L^{\infty}}\Vert u_{n}\Vert^{\frac{p}{2}+1}_{H^{1}(\cG)},
\]
so that (since $p>2$)
\[
\vert I_2\vert=o\left(\Vert u_{n}\Vert_{H^{1}(\cG)}^2 \right).
\]

\textbf{Estimate for $I_2$, case $p\in[4,6)$}: as $p-4\geq0$, there holds
\[
\left\Vert (\vert u_{n}\vert^{2}+\vert v_{n}\vert^{2})^{\frac{p-4}{2}}\right\Vert_{L^{\infty}(\cG)}\leq C.
\]
and then arguing as before one can easily find (as well) that
\[
\left\vert I_2\right\vert = o\left(\Vert u_{n}\Vert^{2}_{H^{1}(\cG)} \right).
\]

Summing up, we proved that for all $p\in(2,6)$,there results
\[
|I_1|+|I_2|=o\left(\Vert u_{n}\Vert^{2}_{H^{1}(\cG)}\right)
\]
and hence, combining with \eqref{PSremainder}, \eqref{sobolevlowerbound} and \eqref{derivativeterm}, we obtain that
\[
\Vert u_{n}\Vert^{2}_{H^{1}(\cG)}=o\left(\Vert u_{n}\Vert^{2}_{H^{1}(\cG)} \right),
\]
which is the contradiction that concludes the proof.
 \end{proof}

 We now prove item (iii).
 \begin{lemma}
  \label{lem-PSNLS}
  The sequence $\left( u_{n}\right)$ is a Palais-Smale sequence for $J$.
 \end{lemma}
 
 \begin{proof}
By Lemma \ref{lem-equiv} it is sufficient to prove \eqref{eq-psaux}. Take, then, $\varphi\in H$ with $\Vert\varphi\Vert_{H^{1}(\cG)}\leq1$. Multiplying \eqref{PSremainder} by $\varphi$ and integrating over $\cG$ (which is rigorous as we showed in the proof of Lemma \ref{lem-nonzero}) one gets
\begin{multline}\label{testPS}
-\int_{\cG}\overline{\varphi}u''_{n}\dx+a_n\int_{\cG}\overline{\varphi}u_{n}\dx=\\[.2cm]
=-\frac{i}{c_{n}}\int_{\cK}\overline{\varphi}\left[(\vert u_{n}\vert^{2}+\vert v_{n}\vert^{2})^{\frac{p-2}{2}}v_{n} \right]'\dx+b_n\int_{\cK}(\vert u_{n}\vert^{2}+\vert v_{n}\vert^{2})^{\frac{p-2}{2}}u_{n}\overline{\varphi}\dx.
\end{multline}
Arguing as in the proof of Lemma \ref{lem-nonzero} and using Lemma \ref{lem-vnzero}, one can check that
\be\label{PS1}
-\int_{\cG}\overline{\varphi}u''_{n}\dx=\int_{\cG}\overline{\varphi}'_{n}u'_{n}\dx+\sum_{\vv\in\cK}\left(\sum_{e\succ\vv}\overline{\varphi}(\vv)\frac{d}{dx}u_{n,e}(\vv) \right)=\int_{\cG}\overline{\varphi}'_{n}u'_{n}\dx+o(1)
\ee
(where throughout we mean that $o(1)$ is independent of $\phi$). Now, the first integral at the r.h.s. of \eqref{testPS} reads
\be\label{PS2}
\begin{split}
\int_{\cK}\overline{\varphi}\left[(\vert u_{n}\vert^{2}+\vert v_{n}\vert^{2})^{\frac{p-2}{2}}v_{n}\right]'\dx&=-\int_{\cK}\overline{\varphi}'\left[(\vert u_{n}\vert^{2}+\vert v_{n}\vert^{2})^{\frac{p-2}{2}}v_{n} \right]\dx\\[.2cm]
&+\sum_{\vv\in\cK}\left(\sum_{e\succ\vv}\overline{\varphi}(\vv)(\vert u_{n,e}(\vv)\vert^{2}+\vert v_{n,e}(\vv)\vert^{2})^{\frac{p-2}{2}}v_{n,e}(\vv)_{\pm} \right),
\end{split}
\ee
where the former term is estimated by
\be
\label{PS3}
\int_{\cK}(\vert u_{n}\vert^{2}+\vert v_{n}\vert^{2})^{\frac{p-2}{2}}\vert v_{n}\vert\vert\overline{\varphi}'\vert\dx\lesssim\int_{\cK}\vert v_{n}\vert\vert\varphi'\vert\dx\lesssim\Vert v_{n}\Vert_{L^{2}(\cK)}\Vert\varphi'\Vert_{L^{2}(\cG)}=o(1),
\ee
whereas the latter is estimated by
\be
\label{PS4}
\sum_{\vv\in\cK}\left(\sum_{e\succ\vv}|\overline{\varphi}(\vv)|(\vert u_{n,e}(\vv)\vert^{2}+\vert v_{n,e}(\vv)\vert^{2})^{\frac{p-2}{2}}|v_{n,e}(\vv)| \right)\lesssim\Vert\varphi\Vert_{L^{\infty}(\cG)}\Vert v_{n}\Vert_{L^{\infty}(\cG)}=o(1),
\ee
(exploiting Lemmas \ref{h1bound} and \ref{lem-vnzero}). It is, then, left to discuss the last term at the r.h.s. of \eqref{testPS}. First note that
\[
 \int_\cK\left(|u_n|^2+|v_n|^2\right)^{\f{p-2}{2}}u_n\overline{\phi}\dx=\int_\cK|u_n|^{p-2}u_n\overline{\phi}\dx+\underbrace{\int_\cK\big[\left(|u_n|^2+|v_n|^2\right)^{\f{p-2}{2}}-|u_n|^{p-2}\big]u_n\overline{\phi}\dx}_{=:R}
\]
and that $\left(|u_n|^2+|v_n|^2\right)^{\f{p-2}{2}}-|u_n|^{p-2}\geq0$.

Let us distinguish two cases (as in the proof of Lemma \ref{lem-nonzero}). Assume first that $p\in(2,4)$. Therefore $0<\frac{p-2}{2}<1$, and this implies that
\begin{multline*}
|R|\leq\int_{\cK}\big[(\vert u_{n}\vert^{2}+\vert v_{n}\vert^{2})^{\frac{p-2}{2}} -\vert u_{n}\vert^{p-2}\big]|u_{n}||\overline{\varphi}|\dx\leq\int_{\cK}\vert v_{n}\vert^{p-2}\vert u_{n}\vert\vert\overline{\varphi}\vert\dx\\[.2cm]
\leq \Vert u_{n}\Vert_{L^{\infty}(\cG)}\Vert v_{n}\Vert_{L^{\infty}(\cG)}^{p-2}\Vert \overline{\varphi}\Vert_{L^{2}(\cG)}=o(1).
\end{multline*}
(where we used again Lemmas \ref{h1bound} and \ref{lem-vnzero}).

On the other hand, assume that $p\in(4,6)$ (the case $p=4$ is analogous). In this case we exploit the elementary inequality
\[
(a+b)^t-a^t\leq c_t\, b\,(a^{t-1}+b^{t-1}),\qquad\forall a,b>0,
\]
with $t>1$ and $c_t>0$. Then, setting $t=\frac{p-2}{2}>1$, $a=\vert u_{n}\vert^{2}$ and $b=\vert v_{n}\vert^{2}$, we have that
\[
\int_{\cK}\big[\left( \vert u_{n}\vert^{2}+\vert v_{n}\vert^{2}\right)^{\frac{p-2}{2}}-\vert u_{n}\vert^{p-2}\big]|u_{n}||\overline{\varphi}|\dx\lesssim\int_{\cK}\left(\vert u_{n}\vert^{\frac{p-4}{2}}+\vert v_{n}\vert^{\frac{p-4}{2}}\right)|v_n|^2|u_n||\phi|\dx=o(1).
\]
Summing up, we proved that for all $p\in(2,6)$ there holds
\[
\int_{\cK}\left( \vert u_{n}\vert^{2}+\vert v_{n}\vert^{2}\right)^{\frac{p-2}{2}} u_{n}\overline{\varphi}\dx=\int_{\cK}\vert u_{n}\vert^{p-2} u_{n}\overline{\varphi}\dx+o(1)
\]
and, combining with \eqref{testPS}, \eqref{PS1}, \eqref{PS2}, \eqref{PS3} and \eqref{PS4}, one gets \eqref{eq-psaux}, which concludes the proof.
\end{proof}

Finally, we have all the ingredients to prove point (iv) and thus Theorem \ref{thm-limit}.

\begin{proof}[Proof of Theorem \ref{thm-limit}]
 From Lemma \ref{lem-PSNLS} the sequence $(u_n)$ is a Palais-Smale sequence for $J$. In addition, from Lemma \ref{h1bound} it is bounded in $H$ so that (up to subsequences)
 \begin{equation}
  \label{eq-conv}
  \left\{\begin{array}{ll}
   \displaystyle u_n\rightharpoondown u & \text{in}\quad H^{1}(\cG),\\[.2cm] 
   \displaystyle u_n\to u               & \text{in}\quad L^p(\cK).
  \end{array}
  \right.
 \end{equation}
 Now, following \cite{ST-JDE}, define the linear functional $B(u):H^{1}(\cG)\to\R$
 \[
  B(u)v:=\int_\cG u'\overline{v}'\dx-\lambda\int_\cG u\overline{v}\dx.
 \]
 From Lemma \ref{lem-PSNLS}, \eqref{eq-bn}, \eqref{eq-an} and \eqref{eq-conv},
 \begin{align*}
  o(1)= & \, \langle A_n(u_n)-B(u)|u_n-u\rangle\\[.2cm]
      = & \, \int_\cG|u_n'-u'|^2\dx-b_n\int_\cK|u_n|^{p-2}u_n(\overline{u}_n-\overline{u})\dx+a_n\int_\cG u_n(\overline{u}_n-\overline{u})\dx+\lambda\int_\cG u(\overline{u}_n-\overline{u})\dx=\\[.2cm]
      = & \, \int_\cG|u_n'-u'|^2\dx-\lambda\int_\cG|u_n-u|^2\dx+o(1),
 \end{align*}
 and, since $\lambda<0$, this entails that $u_n\to u$ in $H^{1}(\cG)$. Since by Lemma \ref{lem-nonzero} $u\neq0$ (recalling also Lemma \ref{lem-vnzero}), the claim of the theorem is proved.
\end{proof}

%%%%%%%%%%%%%%%%%%%%%%%%%%%%%%%%%%%%%%%%%%%%%%%%%%%%%%%%%%%%%%%%%%%%%%%%%%%%%%%%%%%%%%%
%%%%%%%%%%%%%%%%%%%%%%%%%%%%%%%%%%%%%%%%%%%%%%%%%%%%%%%%%%%%%%%%%%%%%%%%%%%%%%%%%%%%%%%
%%%%%%%%%%%%%%%%%%%%%%%%%%%%%%%%%%%%%%%%%%%%%%%%%%%%%%%%%%%%%%%%%%%%%%%%%%%%%%%%%%%%%%%

\appendix
\section{The Dirac operator on metric graphs}
\label{sec-linear}
In this section, we present a sketch of the proof of Proposition \ref{spectrumkirchoff}, that is, self-adjointness and spectrum of the operator $\cD$ introduced by Definition \ref{defi-dirac}.

\medskip
Preliminarily, we observe that, using well-known results from the literature about self-adjoint extensions, one could check that the operator $\cD$ is in fact self-adjoint (see, e.g. \cite{AP-JPA,BT-JMP,P-OM}). In particular, the main result of \cite{BT-JMP} proves self-adjointness for Dirac operators on metric graphs with a wide family of linear vertex conditions (including the Kirchoff-type ones \eqref{eq-kirchtype1}-\eqref{eq-kirchtype2}).

On the other hand, the study the spectral properties of $\cD$ requires some further efforts. First, one has to study the operator on the single components of the graph (segments and halflines) imposing suitable boundary conditions. Then, one describes the effect of connecting these one-dimensional components according to the topology of the graph, through the vertex condtions \eqref{eq-kirchtype1}-\eqref{eq-kirchtype2}.

This can be done, for instance, using the Theory of \emph{Boundary Triplets}. In this section we limit ourselves to a brief presentation of the main ideas and techniques. We refer the reader to \cite{CMP-JDE,GT-p} and references therein for more details.

\medskip
Let us start by recalling some basic notions. Let $A$ be a densely defined closed symmetric operator in a separable Hilbert space $\mathcal{H}$ with equal deficiency
indices $\mathrm{n}_\pm(A):=\dim \cN_{\pm \I} \leq \infty,$ where
$\cN_z:=\ker(A^*-z)$ is the defect subspace.

\begin{defi}
A triplet $\Pi=\{\gH,\gG_0,\gG_1\}$ is called a \emph{Boundary Triplet} for the adjoint operator $A^*$ if and only if $\gH$ is an auxiliary
Hilbert space and $\Gamma_0,\Gamma_1:\  \dom(A^*)\rightarrow \gH$
are linear mappings such that the second abstract Green identity
\[
\langle A^*f|g\rangle - \langle f|A^*g\rangle = \langle\gG_1f,\gG_0g\rangle_\gH-
\langle\gG_0f,\gG_1g\rangle_\gH, \qquad f,g\in\dom(A^*),
\]
holds ($\langle\cdot,\cdot\rangle_\gH$ denoting the scalar product in $\gH$) and the mapping
\[
 \gG:=\begin{pmatrix}\Gamma_0\\[.2cm]\Gamma_1\end{pmatrix}:  \dom(A^*)
\rightarrow \gH \oplus \gH
\]
is surjective.
\end{defi}

\begin{defi}
 Let $\Pi=\{\gH,\gG_0,\gG_1\}$ be a boundary triplet for the adjoint operator $A^*$. Consider, in addition, the operator
 \[
  A_{0}:=A^{*}|_{\operatorname{ker}\Gamma_{0}}
 \]
and denote by $\rho(A_{0})$ its resolvent set. Then, the operator valued functions $\gamma(\cdot) :\rho(A_0)\rightarrow  \mathcal{L}(\gH,\cH)$ and
$M(\cdot):\rho(A_0)\rightarrow  \mathcal{L}(\gH)$ defined by
  \begin{equation}\label{gammafield}
\gamma(z):=\left(\Gamma_0|_{\cN_z}\right)^{-1}
\qquad\text{and}\qquad M(z):=\Gamma_1\circ\gamma(z), \qquad
z\in\rho(A_0),
  \end{equation}
are called the {\em $\gamma$-field} and {\em Weyl function},
respectively, associated with $\Pi.$
\end{defi}

\medskip
Then, we can sketch how to apply the Theory of Boundary Triplets to metric graphs. First, observe that the set $\mE$ of the edges of a metric graph $\cG$ can be decomposed in two subsets: $\mathrm{E}_{s}$ of the bounded edges and $\mathrm{E}_{h}$ of the half-lines. 

Fix, then, $e\in \mathrm{E}_{s}$ and consider the corresponding minimal operator  $\widetilde{D_{e}}$ on $\mathcal{H}_{e}=L^{2}(I_{e})\otimes\C^{2}$, with the same action of \eqref{eq-actionD} and domain $H^{1}_{0}(I_{e})\otimes\C^{2}$. The domain of the adjoint operator, which acts as $\widetilde{D_{e}}$, is 
\[
 \dom(\widetilde{D^{*}_{e}})=H^{1}(I_{e})\otimes\C^{2}
\]
and a suitable choice of trace operators (introduced in \cite{GT-p}) is given by $\Gamma^{e}_{0,1}:H^{1}(I_{e})\otimes\C^{2}\rightarrow\C^{2}$, with
\[
\Gamma^{e}_{0}\begin{pmatrix}\psi^{1}_e \\[.2cm] \psi^{2}_e  \end{pmatrix}=\begin{pmatrix} \psi^{1}_e(0) \\[.2cm] ic\psi^{1}_e(\ell_{e})\end{pmatrix},\qquad \Gamma^{e}_{1}\begin{pmatrix}\psi^{1}_e \\[.2cm] \psi^{2}_e  \end{pmatrix}=\begin{pmatrix} ic\psi^{2}_e(0) \\[.2cm] \psi^{2}_e(\ell_{e})\end{pmatrix}.
\]
In addition, given the boundary triplet $\left\{\gH_{e},\Gamma^{e}_{0},\Gamma^{e}_{1} \right\}$, with $\gH_{e}=\C^{2}$, one can compute the gamma field and the Weyl function using \eqref{gammafield}, and prove that $\widetilde{D^{*}_{e}}$ has defect indices $n_{\pm}(\widetilde{\cD_{e}})=2$. Note, also, that the operator $\cD_e$, with the same action of $\widetilde{D^{*}_{e}}$ and domain
\[
\dom(\cD_{e})=\operatorname{ker}\Gamma^{e}_{0},
\]
is self-adjoint by construction.

Analogously, fix $e'\in \mathrm{E}_{h}$ and consider the minimal operator $\widetilde{\cD_{e'}}$ on $\cH_{e'}=L^{2}(\R_{+})\otimes\C^{2}$, with the same action as before and domain $H^{1}_{0}(\R_{+})\otimes\C^{2}$. The adjoint operator has domain
\[
 \dom(\widetilde{\cD^{*}_{e'}})=H^{1}(\R_{+})\otimes\C^{2}
\]
and the trace operators $\Gamma^{e'}_{0,1}:H^{1}(\R_{*})\otimes\C^{2}\rightarrow\C$ can be defined as
\[
\Gamma^{e'}_{0}\begin{pmatrix}\psi^{1}_{e'} \\[.2cm] \psi^{2}_{e'}  \end{pmatrix}= \psi^{1}_{e'}(0),\qquad \Gamma^{e'}_{1}\begin{pmatrix}\psi^{1}_{e'} \\[.2cm] \psi^{2}_{e'}  \end{pmatrix}= ic\psi^{2}_{e'}(0).
\]
Again, the gamma field and the Weyl function are provided by \eqref{gammafield} (with respect to the boundary triplet $\{ \gH_{e'},\Gamma^{e'}_{0},\Gamma^{e'}_{1}\}$, with $\gH_{e'}=\C$), while the defect indices are $n_{\pm}(\widetilde{\cD_{e'}})=1$. As before, the operator
\[
\cD_{e'}:=\widetilde{\cD^{*}_{e'}},\qquad\dom(\cD_{e'}):=\operatorname{ker}\Gamma^{e}_{0}
\]
is self-adjoint by construction.

As a further step, consider the operator on $\cH=\bigoplus_{e\in\mathrm{E}_{s}}\cH_{e}\oplus\bigoplus_{e'\in\mathrm{E}_{h}}\cH_{e'}$ defined as the direct sum
\[
\cD_{0}:=\bigoplus_{e\in \mathrm{E}_{s}}\cD_{e}\oplus\bigoplus_{e'\in \mathrm{E}_{h}}\cD_{e'},
\]
whose domain is given by the direct sum of the domains of the addends. The spectrum of the operator $\cD_{0}$, it is given by the superposition of the spectra of each addend, that is
\[
\sigma(\cD_{0})=\bigcup_{e\in \mathrm{E}_{s}}\sigma(\cD_{e})\cup\bigcup_{e'\in \mathrm{E}_{h}}\sigma(\cD_{e'}).
\]
Precisely, in \cite{CMP-JDE} it is proved that each segment $I_{e}$, $e\in \mathrm{E}_{s}$ contributes to the point spectrum of $\cD_{0}$ with eigenvalues given by
\begin{equation}\label{pointspectrum}
\sigma(\cD_{e})=\sigma_{p}(\cD_{e})=\left\{\pm\sqrt{\frac{
2mc^2\pi^2}{ \ell_{e}^2}\,\left(j+\frac12\right)^{2}+m^{2}{c^{4}}} \ , \ j\in\mathbb{N}\right\}, \qquad\forall e\in \mathrm{E}_{s},
\end{equation}
while the spectrum on half-lines, on the contrary, is purely absolutely continuous and is given by
\begin{equation}\label{absolutelycontinuous}
\sigma(\cD_{e})=\sigma_{ac}(\cD_{e})=(-\infty,-mc^{2}]\cup[mc^{2},+\infty),\qquad \forall e\in \mathrm{E}_{h}.
\end{equation}

\medskip
Let us describe, now, the Dirac operator introduced in Definition \ref{defi-dirac} using Boundary Triplets. Consider the operator
\[
\widetilde{\cD}:=\bigoplus_{e\in \mathrm{E}_{s}}\widetilde{\cD_{e}}\oplus\bigoplus_{e'\in \mathrm{E}_{h}}\widetilde{\cD_{e'}},
\]
and its adjoint
\[
\widetilde{\cD^{*}}:=\bigoplus_{e\in \mathrm{E}_{s}}\widetilde{\cD^{*}_{e}}\oplus\bigoplus_{e'\in \mathrm{E}_{h}}\widetilde{\cD^{*}_{e'}}
\]
(with obvious definition of the domains). Define, also, the trace operators
\[
 \Gamma_{0,1}=\bigoplus_{e\in \mathrm{E}_{s}}\Gamma^{e}_{0,1}\oplus\bigoplus_{e'\in \mathrm{E}_{h}}\Gamma^{e'}_{0,1}.
\] 
One can prove that $\left\{\gH,\Gamma_{0},\Gamma_{1} \right\}$, with $\gH=\C^{M}$ and $M=2\vert\mathrm{E}_{s}\vert+\vert\mathrm{E}_{h}\vert$,  is a boundary triplet for the operator $\widetilde{\cD^{*}}$, and it is possible to find the corresponding gamma-field and Weyl function arguing as before.

On the other hand, note that boundary conditions \eqref{eq-kirchtype1}-\eqref{eq-kirchtype2} are ``local'', in the sense that at each vertex they are expressed independently of the conditions on other vertices. As a consequence, they can be expressed by means of suitable block diagonal matrices $A,B\in \C^{M\times M}$, with $AB^{*}=BA^{*}$, as
\[
A\Gamma_{0}\psi=B\Gamma_{1}\psi
\]
(the model case at the end of the section clarifies the above notation). Observe also that the sign convention of \eqref{eq-kirchtype2} can be incorporated in the definition of the matrix $B$.

Summing up, the Dirac operator with Kirchoff-type conditions can be defined as 
\[
\cD:=\widetilde{\cD^{*}},\qquad \dom(\cD):=\operatorname{ker}(A\Gamma_{0}-B\Gamma_{1}),
\]
and thus the operator is self-adjoint (again) by construction.

\begin{remark}
The boundary triplets method provides an alternative way to prove the self-adjointness of the Dirac operator with conditions \eqref{eq-kirchtype1}-\eqref{eq-kirchtype2}, different from the classical approach \emph{\`{a} la} Von Neumann adopted in \cite{BT-JMP}.
\end{remark}

It is then left to prove \eqref{eq-sp_D}. As for the Schr\"{o}dinger case \cite{KS-JPA}, the following Krein-type formula for resolvent operators can be proved
\begin{equation}\label{krein}
(\cD-z)^{-1}-(\cD_{0}-z)^{-1}=\gamma(z)\left(B\,M(z)-A\right)^{-1}B\gamma^{*}(\overline{z}), \qquad \forall z\in\rho(\cD)\cap\rho(\cD_{0}),
\end{equation}
and thus the resolvent of the operator $\cD$ can be regarded as a perturbation of the resolvent of the operator $\cD_{0}$. In the above formula $\gamma(\cdot)$ and $M(\cdot)$ are the gamma-field and the Weyl function, respectively, associated with $\cD$ (see \cite{CMP-JDE}). It turns out that the operator appearing in the right-hand side of \eqref{krein} is of finite rank. 
Therefore using Weyl's Theorem \cite[Thm XIII.14]{RS-IV} one can conclude from \eqref{krein} that 
\[
\sigma_{ess}(\cD)=\sigma_{ess}(\cD_{0})=(-\infty,-mc^{2}]\cup[mc^{2},+\infty).
\]

Finally, recall that the point eigenvalues \eqref{pointspectrum} for $\cD_{0}$ are embedded in the continuous spectrum \eqref{absolutelycontinuous}. Hence, in order to conclude the proof of Proposition \ref{spectrumkirchoff} we have to show that they cannot enter the gap $(-mc^2,mc^2)$ as vertex conditions \eqref{eq-kirchtype1}-\eqref{eq-kirchtype2} are imposed.

Let $\lambda\in\sigma(\cD)$ be an eigenvalue. As a consequence, for some $\psi\in \dom(\cD)$, there holds
\[
\cD\psi=\lambda\psi,
\]
that is
\begin{gather}
\label{eigenequation1}-ic\frac{d\psi^{2}}{dx}=(\lambda-mc^{2})\psi^{1}, \\[.2cm]
\label{eigenequation2}-ic\frac{d\psi^{1}}{dx}=(\lambda+mc^{2})\psi^{2}.
\end{gather}
Assuming $\vert\lambda\vert\neq m$, we can divide both sides of \eqref{eigenequation2} by $(\lambda+mc^{2})$ and plug the value of $\psi^{2}$ into \eqref{eigenequation1}, so that
\be\label{eigenlaplace}
-c^{2}\frac{d^{2}\psi^{1}}{dx^{2}}=(\lambda^{2}-m^{2}c^{4})\psi^{1}.
\ee
In addition, combining conditions \eqref{eq-kirchtype1}-\eqref{eq-kirchtype2} yields
\[
\begin{split}
&\sum_{e\succ v}\frac{d\psi^{1}_{e}}{dx}(\vv)=0, \\[.2cm]
&\psi^{1}_{e_{i}}(\vv)=\psi^{1}_{e_{j}}(\vv),\quad\forall e_{i},e_{j}\succ \vv .
\end{split}
\]
Then, $\psi^{1}$ turns out to be an eigenfunction of the laplacian with Kirchhoff vertex conditions on $\cG$. Hence, multiplying \eqref{eigenlaplace} by $\psi^{1}$ and integrating, one can see that 
\[
\vert\lambda\vert > mc^{2},
\]
thus proving that there cannot be any eigenvalue of $\cD$ in $(-mc^{2},mc^{2})$. In other words, imposing Kirchoff-type vertex conditions, the eigenvalues \eqref{pointspectrum} can ``move'' to the thresholds $\pm mc^{2}$, but cannot ``enter the gap''.

%%%%%%%%%%%%%%%%%%%%%%%%%%%%%%%%%%%%%%%%%%%%%%%%%%%%%%%%%%%%%%%%%%%%%%%%%%%%%%%%%%%%%%%%%%%%%%%%%%%%%%%%%%%%%%%%%%%%%%%%%%%%%%

\subsection{A model case: the triple junction}

Let us consider an example in order to clarify the main ideas explained before. Consider a 3-star graph with one bounded edge and two half-lines, as depicted in Figure \ref{fig-grafoesempio}.

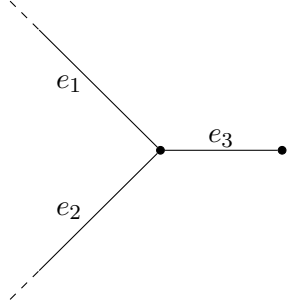
\begin{figure}[h]
 \centering
 \begin{tikzpicture}[xscale= 0.4,yscale=0.4]
 \node at (0,0) [nodo] (00) {};
 \node at (4,0) [nodo] (4) {};
 \node at (-4,4) [infinito] (-44) {};
 \node at (-5,5) [infinito] (-4545) {};
 \node at (-4,-4) [infinito] (-4-4) {};
 \node at (-5,-5) [infinito] (-45-45) {};
 \node at (-3,2.2) [infinito] () {$e_1$};
 \node at (-3,-2) [infinito] () {$e_2$};
 \node at (2,.4) [infinito] () {$e_3$};

 \draw [-] (00) -- (4);
 \draw [-] (00) -- (-44);
 \draw [dashed] (-44) -- (-4545);
 \draw [-] (00) -- (-4-4);
 \draw [dashed] (-4-4) -- (-45-45);
 \end{tikzpicture}
 %%%%%%%%%%%%%%%%%%%%%%%%%%%%%%%%%%%%%%%%%%%%%%%%%
 \caption{A 3-star graph with a finite edge.}
 \label{fig-grafoesempio}
\end{figure}

In this case the finite edge is identified with the interval $I=[0,L]$ and $0$ corresponds to the common vertex of the segment and the half-lines. A suitable choice for the trace operators is
\[
\Gamma_{0}\psi=\left(\begin{array}{c}\psi_{e_1}^{1}(0) \\[.2cm]\psi_{e_2}^{1}(0) \\[.2cm] \psi_{e_3}^{1}(0) \\[.2cm] ic\psi_{e_3}^{1}(L)\end{array}\right),\qquad \Gamma_{1}\psi=\left(\begin{array}{c}ic\psi_{e_1}^{2}(0) \\[.2cm]ic \psi_{e_2}^{2}(0) \\[.2cm]ic \psi_{e_3}^{2}(0) \\[.2cm] \psi_{e_3}^{2}(L)\end{array}\right),
\] 
and the Kirchoff-type conditions \eqref{eq-kirchtype1}-\eqref{eq-kirchtype2} can be rewritten as $A\Gamma_{0}\psi=B\Gamma_{1}\psi$, $AB^{*}=BA^{*}$, where

\[
A=\frac{2}{3}\left(\begin{array}{cccc}-2 & 1 & 1 & 0 \\[.2cm] 1 & -2 & 1 & 0 \\[.2cm] 1 & 1 & -2 & 0 \\[.2cm]0 & 0 & 0 & a\end{array}\right),\qquad B=-\imath\frac{2}{3}\left(\begin{array}{cccc}1 & 1 & 1 & 0 \\[.2cm] 1 & 1 & 1 & 0 \\[.2cm] 1 & 1 & 1 & 0 \\[.2cm]0 & 0 & 0 & b\end{array}\right)
\]
(where, choosing the parameters $a,b\in\mathbb{C}$, we can fix the value of the spinor on the non-connected vertex). Since, as already remarked, conditions \eqref{eq-kirchtype1}-\eqref{eq-kirchtype2} are defined independently on each vertex, one can iterate the above construction for a more general graph structure, thus obtaining matrices $A,B$ with a block structure, each block corresponding to a vertex (for the sake of brevity we omit the details). 

%%%%%%%%%%%%%%%%%%%%%%%%%%%%%%%%%%%%%%%%%%%%%%%%%%%%%%%%%%%%%%%%%%%%%%%%%%%%%%%%%%%%%%%
%%%%%%%%%%%%%%%%%%%%%%%%%%%%%%%%%%%%%%%%%%%%%%%%%%%%%%%%%%%%%%%%%%%%%%%%%%%%%%%%%%%%%%%
%%%%%%%%%%%%%%%%%%%%%%%%%%%%%%%%%%%%%%%%%%%%%%%%%%%%%%%%%%%%%%%%%%%%%%%%%%%%%%%%%%%%%%%

\section{Definition of the form domain}
\label{sec-formdomain}

In Section \ref{sec-quadraticform} we claimed that the form domain of the Dirac operator $\cD$ can be defined interpolating between $L^{2}(\cG,\C^{2})$ and the operator domain \eqref{eq-dirac_domain}. The aim of this section is to provide a more detailed justification of this statement, combining Spectral Theorem and Real Interpolation Theory.

One of the most commonly used forms of the Spectral Theorem states, roughly speaking, that every self-adjoint operator on a Hilbert space is isometric to a multiplication operator on a suitable $L^{2}$-space. In this sense the operator can be "diagonalized" in an abstract way.

\begin{thm}{(\cite[thm. VIII.4]{RS-I})}
Let $H$ be a self-adjoint operator on a separable Hilbert space $\cH$ with domain $\dom(H)$. There exists a measure space $(M,\mu)$, with $\mu$ a finite measure, a unitary operator
\[
U:\cH\longrightarrow L^{2}\left(M,d\mu\right),
\]
and a real valued function $f$ on $M$, a.e. finite, such that
\begin{enumerate}
\item $\psi\in\dom(H)$ if and only if $f(\cdot)(U\psi)(\cdot)\in L^{2}(M,d\mu),$
\item if $\varphi\in U\left(\dom(H) \right)$, then $\left(UHU^{-1}\varphi \right)(m)=f(m)\varphi(m),\quad \forall m\in M$.
\end{enumerate}
\end{thm}
The above theorem essentially says that $H$ is isometric to the multiplication operator by $f$ (still denoted by the same symbol) on the space $L^{2}(M,d\mu)$, whose domain is given by
\[
\dom(f):=\left\{\varphi\in L^{2}(M,d\mu) : f(\cdot)\varphi(\cdot)\in L^{2}(M,d\mu)\right\}.
\]
endowed with the norm
\[
\Vert\varphi\Vert^{2}_{1}:=\int_{M}(1+f(m)^{2})\varphi(m)^{2}d\mu(m)
\]
The form domain of $f$ has an obvious explicit definition, as $f$ is a multiplication operator, that is
\[
 \left\{\varphi\in L^{2}(M,d\mu) : \sqrt{|f(\cdot)|}\,\varphi(\cdot)\in L^{2}(M,d\mu)\right\}
\]

Anyway, it can be recovered using real interpolation theory (we follow the presentation given in \cite{AF,A-p}). Consider the Hilbert spaces $\cH_{0}:=L^{2}(M,d\mu)$ with the norm $\Vert x\Vert_{0}:=\Vert x\Vert_{L^{2}(d\mu)}$, and $\cH_{1}:=\dom(f)$, so that $\cH_{1}\subset\cH_{0}$. Define, in addition, the following quadratic version of Peetre's \emph{K-functional}
\[
K(t,x):=\inf\left\{\Vert x_{0}\Vert^{2}_{0}+t\Vert x_{1}\Vert^{2}_{1} : x=x_{0}+x_{1},x_{0}\in \cH_{0},x_{1}\in \cH_{1}\right\}.
\]
The squared norm $\Vert x\Vert^{2}_{1}$ is a densely defined quadratic form on $\cH_{0}$, represented by
\[
\Vert x\Vert^{2}_{1}=\langle (1+f^{2}(\cdot))x,x\rangle_{0},
\]
where $\langle\cdot,\cdot\rangle_{0}$ is the scalar product of $\cH_{0}$.

By standard arguments (see e.g. \cite{A-p} or \cite[Ch. 7]{AF} and references therein) the intermediate spaces $\cH_{\theta} \subset\left[\cH_{0},\cH_{1}\right]_{\theta}\subset\cH_{0}$, $0<\theta<1$, are given by the elements $x\in\cH_{0}$ such that the following quantity is finite:
\[
\int^{\infty}_{0}\left(t^{-\theta} K(t,x)\right)\frac{dt}{t}<\infty.
\]
Then, for the space $\cH_{\theta}:=\left[\cH_{0},\cH_{1}\right]_{\theta}$ there holds
\[
 \Vert x\Vert^{2}_{\theta}=\langle(1+f^{2}(\cdot))^{\theta}x,x \rangle_{0}.
\]
As a consequence, for $\theta=\frac{1}{2}$ one recovers the form domain of the operator $f$ and, hence, setting $H=\cD$ and $\cH=L^{2}(\cG,\C^{2})$, we can conclude that the space defined in \eqref{interpolateddomain} is exactly the form domain of $\cD$, with $Y=U^{-1}\cH_{\frac{1}{2}}$.

\end{document}